\documentclass[journal]{IEEEtran}

\ifCLASSINFOpdf

\else

\fi

\usepackage[cmex10]{amsmath}
\usepackage{graphics,graphicx,epic,eepic,epsfig,times,units} %
\usepackage{psfrag}
\usepackage{resizegather}		
\usepackage{cite}				
\usepackage{xcolor}				
\usepackage{cases}				
\usepackage{float}

\usepackage{alphalph}
\usepackage{etoolbox}

\AtBeginDocument{%
	\AtBeginEnvironment{subequations}{%
		\let\alph\alphalphval%
	}
}

\usepackage[hyphens]{url}
\usepackage[hidelinks]{hyperref}			
\usepackage{breakurl}
\hypersetup{breaklinks=true}

\usepackage{cleveref}			
\crefformat{equation}{(#2#1#3)}
\crefrangeformat{equation}{(#3#1#4)--(#5#2#6)}
\crefmultiformat{equation}{(#2#1#3)}%
{,~(#2#1#3)}{, (#2#1#3)}{,~(#2#1#3)}

\DeclareMathOperator*{\argmin}{argmin}

\usepackage{tabularx}
\usepackage{multirow}
\usepackage{xcolor}
\usepackage{siunitx}
\DeclareSIUnit \VAr {VAr} 
\DeclareSIUnit \VA {VA} 
\DeclareSIUnit \rad {Radians} 

\newcommand{\reals}{{\mbox{\bf R}}}

\newcommand{\dom}{{\mbox{\bf dom}}}
\newcommand{\epi}{{\mbox{\bf epi}}}

\usepackage{algorithm,algorithmicx}
\usepackage{algpseudocode}
\usepackage{algcompatible}
\algrenewcommand\alglinenumber[1]{\scriptsize #1:}
\newcommand{\algrule}[1][.2pt]{\par\vskip.5\baselineskip\hrule height #1\par\vskip.5\baselineskip}
\algrenewcommand\algorithmicindent{2.0em}%
\usepackage{changepage}

\usepackage{mathtools,nccmath,bm}

\DeclarePairedDelimiter{\nint}\lfloor\rceil

\usepackage[font=small,skip=0pt]{caption}

\usepackage{amsthm,amssymb}
\newtheorem{theorem}{Theorem}

\newtheorem{definition}{Definition}
\let\olddefinition\definition
\renewcommand{\definition}{\olddefinition\normalfont}

\allowdisplaybreaks

\begin{document}

\title{Iterative LP-based Methods for the Multiperiod Optimal Electricity and Gas Flow Problem}

\author{\IEEEauthorblockN{Sleiman~Mhanna,~\emph{Member, IEEE}, Isam Saedi,~\emph{Student~Member, IEEE},
		Pierluigi Mancarella,~\emph{Senior~Member, IEEE}}}

\maketitle

\begin{abstract}
	In light of the increasing coupling between electricity and gas networks, this paper introduces two novel iterative methods for efficiently solving the multiperiod optimal electricity and gas flow (MOEGF) problem. The first is an iterative MILP-based method and the second is an iterative LP-based method with an elaborate procedure for ensuring an integral solution. The convergence of the two approaches is founded on two key features. The first is a penalty term with a single, \emph{automatically tuned}, parameter for controlling the step size of the gas network iterates. The second is a sequence of supporting hyperplanes together with an increasing number of carefully constructed halfspaces for controlling the convergence of the electricity network iterates. Moreover, the two proposed algorithms use as a warm start the solution from a novel polyhedral relaxation of the MOEGF problem, for a noticeable improvement in computation time as compared to a cold start. Unlike the first method, which invokes a branch-and-bound algorithm to find an integral solution, the second method implements an elaborate \emph{steering} procedure that guides the continuous variables to take integral values at the solution. Numerical evaluation demonstrates that the two proposed methods can converge to high-quality \emph{feasible} solutions in computation times at least \emph{two orders of magnitude faster} than both a state-of-the-art nonlinear branch-and-bound (NLBB) MINLP solver and a mixed-integer convex programming (MICP) relaxation of the MOEGF problem. The experimental setup consists of five test cases, three of which involve the \emph{real} electricity and gas transmission networks of the state of Victoria with \emph{actual} linepack and demand profiles.
\end{abstract}

\IEEEpeerreviewmaketitle

\setlength{\belowdisplayskip}{0.5pt} \setlength{\belowdisplayshortskip}{0.5pt}
\setlength{\abovedisplayskip}{0.5pt} \setlength{\abovedisplayshortskip}{0.5pt}
\section*{Notation}
\addcontentsline{toc}{section}{Notation}
\subsection{Sets}
	\begin{IEEEdescription}[\IEEEsetlabelwidth{$\phi_{mn}^{{\rm in/o},t}$}\IEEEusemathlabelsep]
		\item[$\mathcal{B}$] Set of buses in the electricity network.
		\item[$\mathcal{B}_{i}$] Set of buses adjacent to bus $i$.
		\item[$\mathcal{C}$] Set of compressors in the gas network.
		\item[$\mathcal{D}^{\rm g}$] Set of gas-powered generators (GPGs) on the gas network.
		\item[$\mathcal{D}_{m}^{\rm g}$] Set of GPGs connected to node $m$.
		\item[$\mathcal{E}$] Set of non-pipe elements.
		\item[$\mathcal{G}^{\rm g/ng}$] Set of GPGs/non-GPGs in the electricity network.
		\item[$\mathcal{G}_{i}^{\rm g/ng}$] Set of GPGs/non-GPGs connected to bus $i$.
		\item[$\mathcal{L}/\mathcal{L}^{\rm t}$] Set of all branches $ij$/$ji$ where $i$/$j$ is the ``from''/``to'' bus.
		\item[$\mathcal{N}$] Set of nodes in the gas network.
		\item[$\mathcal{P}$] Set of pipelines in the gas network.
		\item[$\mathcal{R}$] Set of pressure regulators in the gas network.
		\item[$\mathcal{S}$] Set of gas supplies in the gas network.
		\item[$\mathcal{S}_{m}$] Set of gas supplies connected to node $m$.
		\item[$\mathcal{T}$] Decision time horizon.
	\end{IEEEdescription}
\subsection{Parameters and input data}
	\begin{IEEEdescription}[\IEEEsetlabelwidth{$\phi_{mn}^{{\rm in/o},t}$}\IEEEusemathlabelsep]
		\item[$c_{0,gi}$] Constant coefficient ($\SI{}{\$\per\hour}$) term of GPG/non-GPG $g$'s cost function.
		\item[$c_{1,gi}$] Coefficient ($\SI{}{\$\per\mega\watt\hour}$) of the linear term of GPG/non-GPG $g$'s cost function.
		\item[$c_{2,gi}$] Coefficient ($\SI{}{\$\per\mega\watt\hour\squared}$) of the quadratic term of GPG/non-GPG $g$'s cost function.
		\item[$c_{sm}$] Cost ($\SI{}{\$\per\meter\cubed}$) of gas production of gas supply $s$ at node $m$.
		\item[$\Delta \tau,\Delta t$] Time resolution ($\SI{}{\second},\SI{}{\hour}$).
		\item[$D_{mn}$] Diameter ($\SI{}{\meter}$) of pipeline $mn$.
		\item[$g^{\rm sh}_{i}$] Shunt conductance (pu) at bus $i$.
		\item[$\mathrm{i}$] Imaginary unit.
		\item[$k$] Iteration number.
		\item[$\eta_{gi}$] Efficiency of GPG $g$ at bus $i$.
		\item[$\phi_{m}^{{\rm d},t}$] Gas demand ($\SI{}{\meter\cubed\per\second}$) at node $m$.
		\item[$L_{mn}$] Length ($\SI{}{\meter}$) of pipeline $mn$.
		\item[$HHV$] Higher heating value ($\SI{38.07}{\mega\joule\per\meter\cubed}$) of natural gas.
		\item[$\nu_{mn}$] Gas consumption coefficient of compressor $mn$.
		\item[$p_{i}^{{\rm d},t}$] Active power demand ($\SI{}{\mega\watt}$) at bus $i$.
		\item[$p_{gi}^{\rm RD/RU}$] Ramp down/up rate ($\SI{}{\mega\watt}/\SI{}{\hour}$) of GPG/non-GPG $g$ at bus $i$.
		\item[$R$] Specific gas constant ($\SI{478.42}{\joule\per\kilogram\per\kelvin}$) at standard conditions.
		\item[$\rho$] Gas density ($\SI{0.735}{\kilogram\per\meter\cubed}$) at standard conditions.
		\item[$\overline{\gamma}_{mn}/\underline{\gamma}_{mn}$] Upper/Lower limit on the compression (pressure drop) ratio of a compressor (pressure regulator).
		\item[$S$] Specific gravity (relative density) of natural gas, $S = 0.6$ (dimensionless).
		\item[$T$] Gas temperature ($\SI{288.15}{\kelvin}$) at standard conditions.
		\item[$T_{ij}$] Complex tap ratio of a phase shifting transformer ($T_{ij}=\tau_{ij}\mathrm{e}^{\mathrm{i} \theta_{ij}^{\rm shift}}$).
		\item[$Y_{ij}$] Series admittance (pu) in the $\pi$-model of branch $ij$.
		\item[$Z_{mn}$] Compressibility factor (dimensionless) of gas in pipeline $mn$.
	\end{IEEEdescription}
\subsection{Operators}
	\begin{IEEEdescription}[\IEEEsetlabelwidth{$\phi_{mn}^{{\rm in/o},t}$}\IEEEusemathlabelsep]
		\item[$\bullet^*$] Conjugate operator.
		\item[$\Im/\Re\left\{\bullet\right\}$] Imaginary/Real value operator.
		\item[$\underline{\bullet}/\overline{\bullet}$] Minimum/Maximum magnitude operator.
		\item[$\left|\bullet\right|$] Magnitude operator/Cardinality of a set.
		\item[$\nint{\bullet}$] Round-to-nearest-integer operator.
		\item[$\times$] Cross product of two vectors.
	\end{IEEEdescription}
\subsection{Variables}
	\begin{IEEEdescription}[\IEEEsetlabelwidth{$\phi_{mn}^{{\rm in/o},t}$}\IEEEusemathlabelsep]
		\item[$p_{gi}^{t}$] Active power ($\SI{}{\mega\watt}$) generation of GPG/non-GPG $g$ at bus $i$.
		\item[$p_{ij}^{t}$] Active power ($\SI{}{\mega\watt}$) flow along branch $ij$.
		\item[$\theta_{i}^{t}$] Voltage angle ($\SI{}{\radian}$) at bus $i$.
		\item[$\wp_{m}^{t}$] Gas pressure ($\SI{}{\pascal}$) at node $m$.
		\item[$\tilde{\wp}_{mn}^{t}$] Average gas pressure ($\SI{}{\pascal}$) across pipeline $mn$.
		\item[$\phi_{mn}^{t}$] Gas flow rate ($\SI{}{\meter\cubed\per\second}$) across edge $mn$.
		\item[$\phi_{mn}^{{\rm in/out},t}$] Inflow/Outflow rate of gas ($\SI{}{\meter\cubed\per\second}$) of pipeline $mn$.
		\item[$\phi_{mn}^{{\rm c},t}$] Gas trapped in compressor $mn$ when it is boosting pressure.
		\item[$\phi_{mn}^{+,t}$] Gas flowing across compressor $mn$ when it is operating in the direction of pressure boost.
		\item[$\phi_{mn}^{-,t}$] Gas flowing across compressor $mn$ when it is operating in a reverse direction.
		\item[$\phi_{gm}^{t}$] Gas consumption ($\SI{}{\meter\cubed\per\second}$) of GPG $g$ at node $m$.
		\item[$\phi_{sm}^{t}$] Gas flow rate ($\SI{}{\meter\cubed\per\second}$) from gas supply $s$ at node $m$.
		\item[$\ell_{mn}^{t}$] Linepack ($\SI{}{\meter\cubed}$) of pipeline $mn$.
		\item[$z_{mn}^{t}$] Binary variables for non-pipe elements (also for pipelines in the MICP relaxation).	
	\end{IEEEdescription}

\section{Introduction}

	\IEEEPARstart{T}{raditionally}, energy systems like electricity and gas were modeled, operated, and managed separately. However, recent developments in power-to-gas (PtG) and co-generation technologies, and the heavy reliance on gas-powered generators (GPGs) to balance intermittent generation from renewable energy sources, have all prompted a paradigm shift towards \emph{jointly} modeling and operating electricity and gas systems \cite{Geidl2007_OPFofMultipleEnergyCarriers}. The main sources of complexity in the modeling of integrated electricity and gas systems (IEGS) are the alternating current (AC) power flow equations in the electricity transmission network, the equations describing the dynamic behavior of gas flow in pipelines, and the disjoint sets describing the operation of non-pipe elements such as compressors and pressure regulators in the gas transmission network. This modeling of IEGS results in an NP-hard MINLP problem that has a nonconvex continuous relaxation, which is extremely challenging to solve, even to local optimality, using current state-of-the-art MINLP technology.
	Therefore, tractable alternatives such as linear programming (LP) approximations \cite{Qiu2016_LPforCoPlanning,Luo2018_FullyLinearIEGS,Fang2019_LPforIEGS}, second-order cone programming (SOCP) relaxations \cite{Wang2018_ConvexOGPF,Chen2019_UCwithNG,Yang2020_TwoStageOEGF}, semidefinite programming (SDP) relaxations \cite{Manshadi2018_SDPforGas}, mixed-integer linear programming (MILP) approximations \cite{Correaposada2014_PWLmodelsinGas,CorreaPosada2015_SCUCwithDynamicGas,Zhang2016_HourlyDRinIEGS,He2017_RobustCooptimizationofIEGSwithADMM,Shao2017_MILPOPFinIEGS,Chen2018_OPGFwithlimitedcontrolactions,Zhang2018_MILPforIEGS,Wu2019_distributionallyRobustIEGS}, and mixed-integer second-order cone programming (MISOCP) relaxations \cite{BorrazSanchez2016_ConvexRelaxationsforGas,Wen2018_SynergisticEandGviaADMM,He2018_DecentralizedIEGSwithCones,Singh2019_MISOCPforOGF,Schwele2019_ConvexificationforLinepack,Yang2020_MICPwithEnvelopeTightening,Sayed2020_OEGFSequentialMISOCP,Yang2020_DRCCforOEGF}, are garnering considerable attention in the research community.
	
	
	However, unless realistic approximations and assumptions are involved, achieving tractability may come at the price of infeasibility.
	On the gas network side, one common oversimplification is to adopt a steady-state gas flow model instead of a dynamic gas flow model. In reality, the behavior of gas flow is characterized by much slower dynamics as compared to electric power flow, which means that gas demand and supply are not balanced instantaneously. These slow dynamics give rise to the \emph{linepack}, which is instrumental in assessing gas network flexibility \cite{Clegg2016_IntegratedEandGFlexAssess}.\footnote{The linepack is the volume of gas that can be stored in a pipeline.} In contrast, steady-state gas flow assumes that supply and demand are balanced instantaneously, which does not account for the effect of the linepack, and therefore leads to an overestimation of the gas injections from gas suppliers \cite{Yang2018_EffectofNGinRobustScheduling}.
	
	Another common unrealistic assumption is related to the direction of flow in pipes and non-pipe elements. The assumption of known direction of gas flow in pipe and non-pipe elements is conducive for two main reasons. First, it obviates the disjoint sets describing the operation of non-pipe elements. Second, it renders the equation of gas flow in a pipe easier to convexify or approximate \cite{Wang2018_ConvexOGPF,Chen2019_UCwithNG,Yang2020_TwoStageOEGF}. This assumption is valid under certain conditions where the demand predictably fluctuates within the day. However, in a multiperiod setting, this assumption would no longer be valid in the planning and operation of future IEGS under different clean-fuel (e.g., hydrogen, synthetic methane, etc.) injection and storage scenarios \cite{Clegg2015_IEGSwithPtG}, and different electrification versus decarbonization scenarios of the heating and transportation sectors \cite{Clegg2018_IEGSwithheatpartII,Saedi2020_ElectrificationofHeating}. Without any assumptions on the direction of flow in a pipe, there are two approaches to approximate or convexify the underlying nonconvex equality constraint. The first is through piecewise linear (PWL) techniques, which can be divided into two types. In the first type, each univariate term is approximated by PWL segments \cite{Correaposada2014_PWLmodelsinGas,CorreaPosada2015_SCUCwithDynamicGas,He2017_RobustCooptimizationofIEGSwithADMM,Chen2018_OPGFwithlimitedcontrolactions,Zhang2018_MILPforIEGS}. The second type consists of higher dimensional PWL techniques, such as the 2-D grid triangulation domains (in the 3-D Euclidean space) in \cite{Zhang2016_HourlyDRinIEGS}, and the 3-D PWL techniques which are based on the Taylor series expansion of the 3-variable equation of gas flow in pipelines \cite{Shao2017_MILPOPFinIEGS,Wu2019_distributionallyRobustIEGS}. The aim of these methods is to approximate the original MINLP problem by the more appealing MILP, as it is computationally more efficient and provides a measure of optimality. However, even for a modest approximation accuracy, these MILP approximations would require a computationally prohibitive number of PWL segments \cite{Huchette2017_PWL}.\footnote{This claim is also true for PWL methods employing special ordered sets of type 2 (SOS2). SOS2 constraints, which are ordered sets of variables where at most two variables in the set may take non-zero values, instruct the branch and bound algorithm to branch on sets of variables, rather than individual variables. Examples of the latter include PWL methods that use binary variables for each segment or a logarithmic number of binary variables. A numerical comparison of state-of-the-art PWL methods to the proposed algorithms in this paper is given in \cite{IEGS_data} as part of this work.} The second approach consists of relaxing the nonconvex constraint to an MISOCP one \cite{BorrazSanchez2016_ConvexRelaxationsforGas,Wen2018_SynergisticEandGviaADMM,Singh2019_MISOCPforOGF,Schwele2019_ConvexificationforLinepack,Yang2020_MICPwithEnvelopeTightening}. However, these methods can still result in an infeasible solution as the original nonconvex equality constraint is relaxed into two disjoint SOC inequality constraints. In other words, the MISOCP approach is only feasible if one of the two SOC constraints is active at the optimum. One way to recover a feasible solution is to implement the \emph{convex-concave} procedure \cite{Lipp2016_CCP}, which in this setting entails solving a series of MISOCP problems \cite{He2018_DecentralizedIEGSwithCones,Sayed2020_OEGFSequentialMISOCP,Yang2020_DRCCforOEGF}, but now at the expense of more computation time.
	
	On the electricity network side, common practice is to approximate the AC power flow model by a DC power flow model (see \cite{CongLiu2009_SCUCwithNGconstraints,CorreaPosada2015_SCUCwithDynamicGas,Alabdulwahab2017_StochasticSCIEGS,Wu2019_distributionallyRobustIEGS,Zhao2018_CoordinatedExpansionPlanningofIEGS,He2018_RobustCoPlanningwithReliability,Fang2019_LPforIEGS,He2018_DecentralizedIEGSwithCones,Wen2018_SynergisticEandGviaADMM,Zhang2016_HourlyDRinIEGS,He2017_RobustCooptimizationofIEGSwithADMM,Shao2017_MILPOPFinIEGS,Chen2018_OPGFwithlimitedcontrolactions,Zhang2018_MILPforIEGS,Wang2019_RiskBasedDistributionallyRobustOGPF,Schwele2019_ConvexificationforLinepack,Yang2018_EffectofNGinRobustScheduling}), as it is linear and therefore amenable to convex programming, MILP, and MISOCP. The main drawback of the DC power flow model is its inability to capture losses in transmission lines (and transformers), which may result in an oversimplification of the problem. One way to include these losses is to approximate the cosine terms in the AC power flow constraints by their second-order Maclaurin series \cite{Qiu2016_LPforCoPlanning,Luo2018_FullyLinearIEGS,Yang2020_TwoStageOEGF,Yang2020_MICPwithEnvelopeTightening}. 
	
	Against this background, this paper proposes two \emph{iterative} LP-based methods to directly solve the multiperiod optimal electricity and gas flow (MOEGF) problem with \emph{quasi-dynamic} gas constraints,\footnote{In a quasi-dynamic gas flow model, the continuity equation and the motion equation are discretized over the full length of the pipe, whereby the input flow and the output flow are different. In contrast, a steady-state model assumes constant flow across a pipe, i.e., the input and output flows are equal. As a consequence, the steady-state model fails to capture the linepack, which results in unrealistic gas supply profiles and pressures across the network. Interested readers are referred to \cite{CorreaPosada2015_IEGSandAdequacy} and \cite{Chaudry2008_MuliperiodGas&Electricity} for more details on the derivation of the quasi-dynamic gas flow model.} \emph{strengthened} DC power flow, bidirectional pipes, compressors, and pressure regulators. Without the procedures that address the integrality constraints, the two methods are similar in concept to sequential linear programming (SLP), which has been applied to the optimal power flow (OPF) problem in \cite{Castillo2016_SLPforOPF}, and the gas transmission problem in \cite{ONeill1979_MPforallocationofNG}, but has not been applied to the MOEGF problem with quasi-dynamic gas constraints and disjoint sets. The step size in \cite{Castillo2016_SLPforOPF} and \cite{ONeill1979_MPforallocationofNG} is controlled by tightening the bounds on selected variables at each iteration, which requires multiple parameters and problem-dependent tuning, and can result in infeasible LP subproblems. In contrast, in this work, the step size for the gas network iterates is controlled in the objective function by a penalty term with a single, automatically tuned, parameter. More interestingly, electricity network iterates are not controlled by a step size but by dynamically introducing affine and linear cuts as a sequence of supporting hyperplanes and an increasing number of carefully constructed supporting halfspaces. Relevant previous works involving LP consisted mainly of first-order Taylor series approximations around a fixed pre-determined operating point \cite{Qiu2016_LPforCoPlanning,Luo2018_FullyLinearIEGS,Fang2019_LPforIEGS}. However, although computationally efficient, those works do not yield feasible solutions.
	
	In a nutshell, this paper advances the state of the art in the following ways:
	\begin{itemize}
		\item It presents a novel iterative MILP-based method that is numerically demonstrated to converge to near-optimal and feasible solutions in at least one order of magnitude faster than both a state-of-the-art NLBB MINLP solver and a mixed-integer convex programming (MICP) relaxation of the problem.  
		\item It presents a novel fast iterative LP-based method that is numerically demonstrated to converge to high-quality feasible solutions in at least \emph{two orders of magnitude faster} than a state-of-the-art NLBB MINLP solver. Unlike the iterative MILP-based method which invokes a branch-and-bound algorithm to find an integral solution, this approach implements a fast iterative LP-based \emph{steering} procedure that guides the continuous variables to take integral values at the solution.
		\item It presents novel polyhedral envelopes for the direction of flow and average pressure in a pipeline and proves that the nonconvex equality constraint describing the average pressure in a pipeline becomes convex when relaxed into an inequality constraint.
		\item It numerically shows that the two proposed methods are robust against the choice of starting point but warm-starting them from the solution of the proposed polyhedral relaxation substantially improves computational speed. The proposed polyhedral relaxation also provides a valid starting point when no prior information is available.
	\end{itemize}
	
	While convex relaxations such as the SOCP, SDP, and the MISOCP are necessary to better understand the structure and complexity of the problem, they generally do not yield feasible solutions. Moreover, existing SOCP relaxations in \cite{Wang2018_ConvexOGPF,Chen2019_UCwithNG,Yang2020_TwoStageOEGF} can only be applied when the direction of flow in pipeline is known beforehand, and the existing SDP relaxation in \cite{Manshadi2018_SDPforGas} can only be applied in steady-state modeling where the square of the pressure terms can be substituted by linear ones. On the other hand, the MISOCP relaxation in the transient domain in \cite{BorrazSanchez2016_ConvexRelaxationsforGas,Wen2018_SynergisticEandGviaADMM,He2018_DecentralizedIEGSwithCones,Singh2019_MISOCPforOGF,Schwele2019_ConvexificationforLinepack,Yang2020_MICPwithEnvelopeTightening}, although able to model bi-directional gas flow in pipes, also yields infeasible solutions in general. This, along with the heavy computational burden of the MISOCP relaxation, are numerically confirmed in Section~\ref{Sec_NumericalEval} of this paper on real-world systems. Nonetheless, the MICP-relaxed problem provides a good-quality \emph{lower bound} than can be used to assess the optimality of the two proposed methods. In contrast, the two proposed method are not convex relaxations, but are rather founded on a computationally efficient SLP method that exploits the inherent engineering limits dictated by the design parameters of components, as well as the structure and other properties of the MOEGF problem with quasi-dynamic gas flow modeling, to obtain high-quality \emph{feasible} solutions. The validity of both the modeling and the solutions are demonstrated on the real Victorian gas transmission system with actual linepack and pressure profiles from the Australian Energy Market Operator (AEMO) \cite{AEMO}.
	
	The purpose and innovation of the proposed two iterative algorithms are two fold. First, they leverage the superior computational efficiency of state-of-the-art MILP and LP solvers, while retaining the accuracy of interior-point methods (IPM). Second, since the market dispatch engines of most Independent System Operators (ISO) around the world use MILP or LP solvers to obtain the locational marginal prices (LMP), by solving a problem that incorporates an LP approximation of the OPF problem, this approach is designed to allow ISO to retain their MILP or LP solvers but now with the added benefits of capturing the couplings between electricity and gas networks, gas network flexibility, and LMP that reflect transmission line losses, all in an integrated optimization framework.
	
	It is worth noting that current electricity and gas markets are still operated separately. However, extensive recent literature has shown that an integrated operation of both systems leads to notable operational cost savings and a better quantification of the flexibility of both systems \cite{Geidl2007_OPFofMultipleEnergyCarriers,CorreaPosada2015_IEGSandAdequacy,Chaudry2008_MuliperiodGas&Electricity,Ameli2019_CoordinatedOperationofGE,Clegg2016_IntegratedEandGFlexAssess}. In addition, the advent of power-to-gas technologies will further increase the interactions between the two networks, thereby reinforcing the need for an integrated modeling \cite{Clegg2015_IEGSwithPtG}. Nonetheless, the concepts underlying the proposed algorithms in this paper can still be straightforwardly applied to each market/system independently.
	
	The paper is organized as follows. Section \ref{Sec_MOEGF} describes the MOEGF problem and Section \ref{Sec_MICP} presents an MICP relaxation and an MISOCP relaxation of the MOEGF problem. Section \ref{Sec_PolyEnv} introduces the polyhedral envelopes of the nonconvex sets, and the resulting polyhedral relaxation of the MOEGF problem. The proposed iterative MILP-based method and the iterative LP-based method are introduced in Sections \ref{Sec_SMILP} and \ref{Sec_SLPH}, respectively. The optimality, feasbility, and computational effort of the two proposed methods are compared to those of a state-of-the-art NLBB solver and the MICP relaxation of the MOEGF problem in Section \ref{Sec_NumericalEval}, which also assesses the output linepack of the two proposed methods against historical ones from AEMO. The paper concludes in Section \ref{Sec_Conclusion}.

\section{Multiperiod Optimal electricity and gas flow}\label{Sec_MOEGF}

	The MOEGF problem consists of finding the least-cost dispatch of power from electric generators, and gas from gas supplies, to satisfy both electrical and gas demands at all buses in the electrical network and nodes in the gas network. Electric power flow is governed by physical laws such as Ohm's law and Kirchhoff's current law (KCL), and other technical restrictions. On the other hand, the flow of gas is governed by physical laws, such as the conservation of flow at each node, the \emph{quasi-dynamic} behavior of gas flow in pipelines, and other operational requirements such as valve switching in compressor and pressure regulation stations. The MOEGF can be mathematically formulated as
	\begin{align}
		\hspace{-1.25cm} \underset{\substack{p_{gi}^{t},p_{ij}^{t},\theta_{i}^{t},\wp_{m}^{t}\\\tilde{\wp}_{mn}^{t},\phi_{mn}^{t},\ell_{mn}^{t},\phi_{mn}^{{\rm in},t}\\\phi_{nm}^{{\rm out},t},z_{mn}^{t},\phi_{gm}^{t},\phi_{sm}^{t}\\\phi_{nm}^{{\rm c},t},\phi_{nm}^{+,t},\phi_{nm}^{-,t}}} 
		{\mbox{ minimize}} & \sum_{t \in \mathcal{T}} \Bigg( \sum_{gi \in \mathcal{G}} \hspace{-0cm} f_{gi}\left(p_{gi}^{t}\right)\Bigg. + \hspace{-0.2cm} \hspace{-0cm} & \hspace{-0cm} & \nonumber
	\end{align}
	\vspace{-1.0cm}
	\begin{subequations}\label{OEGF}
		\begin{align}
			& \hspace{0cm} \Bigg. \sum_{sm\in \mathcal{S}} f_{sm}\left(\phi_{sm}^{t}\right)\Bigg) \hspace{-3cm} & \hspace{-0cm} & \label{OEGF_objective} \\
			\text{subject to } \ \ \underline{p}_{gi} \leq p_{gi}^{t} & \leq \overline{p}_{gi}, \hspace{-1.5cm} & \hspace{-0cm} gi \in \mathcal{G} \label{OEGF_Pminmax} \\
			- p_{gi}^{\rm RD} \Delta t \le p_{gi}^{t} - p_{gi}^{t-1} & \leq p_{gi}^{\rm RU} \Delta t, \hspace{-1.5cm} & \ \ \quad \quad gi \in \mathcal{G} \label{OEGF_Pramp} \\
			-\overline{p}_{ij} \leq p_{ij}^{t} & \leq \overline{p}_{ij}, \hspace{-1.5cm} & ij \in \mathcal{L} \cup \mathcal{L}&^{\rm t} \label{OEGF_MWlimits} \\
			\theta_{ij}^{t} = \theta_{i}^{t}-\theta_{j}^{t}, \ \underline{\theta}_{ij} & \leq \theta_{ij}^{t} \leq \overline{\theta}_{ij}, \hspace{-1.5cm} & ij \in \mathcal{L} & \label{OEGF_anglediff} \\
			\sum_{g \in \mathcal{G}_{i}^{\rm g} \cup \mathcal{G}_{i}^{\rm ng}} p_{gi}^{t} = p_{i}^{{\rm d},t} & + \sum_{j \in \mathcal{B}_{i}} p_{ij}^{t} + g^{\rm sh}_{i}, \hspace{-1.5cm} & i \in \mathcal{B} & \label{OEGF_KCLp} \\
			p_{ij}^{t}= g_{ij} \hspace{0.5cm} & \hspace{-0.5cm} 0.5(\theta_{ij}^{t})^2 + b_{ij}\theta_{ij}^{t}, \hspace{-1.5cm} & ij \in \mathcal{L} & \label{OEGF_Pij} \\
			p_{ji}^{t}= g_{ji} \hspace{0.5cm} & \hspace{-0.5cm} 0.5(\theta_{ij}^{t})^2 - b_{ji}\theta_{ij}^{t}, \hspace{-1.5cm} & ji \in \mathcal{L}&^{\rm t} \hspace{-0.1cm} \label{OEGF_Pji} \\
			p_{gi}^{t}=\phi_{gm}^{t} HHV &\eta_{gi}, & \hspace{-3cm} gi \in \mathcal{G}^{\rm g}, \ gm \in \mathcal{D}&^{\rm g} \label{OEGF_gasflow_GFG}\\ 
			\underline{\phi}_{sm}^{\rm s} \leq \phi_{sm}^{t} & \leq \overline{\phi}_{sm}^{\rm s}, \hspace{-0.5cm} & sm \in \mathcal{S}& \label{OEGF_Fminmax} \\
			\underline{\wp}_{m} \leq \wp_{m}^{t} & \leq \overline{\wp}_{m} & m \in \mathcal{N} & \label{OEGF_PressureMinMax}\\
			&\hspace{-3.45cm} \sum_{s \in \mathcal{S}_{m}} \phi_{sm}^{t} = \sum_{mn \in \mathcal{P}} \phi_{mn}^{{\rm in},t} - \sum_{nm \in \mathcal{P}} \phi_{nm}^{{\rm out},t} + \hspace{-1.5cm} & & \nonumber \\ 
			&\hspace{-3.5cm} \sum_{mn \in \mathcal{E}} \phi_{mn}^{t} - \hspace{-0.1cm} \sum_{nm \in \mathcal{E}} \phi_{nm}^{t} + \hspace{-0.1cm} \sum_{g \in \mathcal{D}_{m}^{\rm g}} \phi_{gm}^{t} + \phi_{m}^{{\rm d},t} , \hspace{-1.5cm} & m \in \mathcal{N} & \label{OEGF_Nodal} \\
			&\hspace{-3cm} \phi_{mn}^{t}\left|\phi_{mn}^{t} \right| = \Phi_{mn}\left((\wp_{m}^{t})^2 - (\wp_{n}^{t})^2 \right), \hspace{-1.5cm} & mn \in \mathcal{P} & \label{OEGF_gasflow_pipes}\\
			\phi_{mn}^{t} = 0.5 & \left(\phi_{mn}^{{\rm in},t} + \phi_{mn}^{{\rm out},t} \right), \hspace{-1.5cm} & mn \in \mathcal{P} & \label{OEGF_averageflow}\\ 
			&\hspace{-3cm} \tilde{\wp}_{mn}^{t} = \frac{2}{3} \left(\wp_{m}^{t} + \wp_{n}^{t} - \frac{\wp_{m}^{t}\wp_{n}^{t}}{\wp_{m}^{t} + \wp_{n}^{t}} \right), \hspace{-1.5cm} & mn \in \mathcal{P} & \label{OEGF_averagepressure} \\
			\ell_{mn}^{t}=&\Psi_{mn}\tilde{\wp}_{mn}^{t}, & mn \in \mathcal{P} & \label{OEGF_linepack1} \\
			\ell_{mn}^{t}=\ell_{mn}^{t-1} + \Delta \tau & \left( \phi_{mn}^{{\rm in},t} - \phi_{mn}^{{\rm out},t} \right), \hspace{-1.5cm} & mn \in \mathcal{P} & \label{OEGF_continuity} \\
			-\overline{\phi}_{mn} \leq \phi_{mn}^{t},\phi_{mn}^{{\rm in},t}&,\phi_{mn}^{{\rm out},t} \leq \overline{\phi}_{mn}, \hspace{-1.5cm} & mn \in \mathcal{P} & \label{OEGF_FlowPipeMinMax}\\
			&\hspace{-1.5cm} \phi_{mn}^{t} = \phi_{mn}^{+,t} + \phi_{mn}^{-,t}, \hspace{-1.5cm} & mn \in \mathcal{C} & \label{OEGF_compflowplusminus} \\
			&\hspace{-1.5cm} \phi_{mn}^{{\rm c},t} = \nu_{mn} \phi_{mn}^{+,t} , \hspace{-1.5cm} & mn \in \mathcal{C} & \label{OEGF_compgastrap} \\
			&\hspace{-2cm} 0 \le \phi_{mn}^{+,t} \le z_{mn}^{t}\overline{\phi}_{mn} , \hspace{-1.5cm} & mn \in \mathcal{C} & \label{OEGF_compplus} \\
			&\hspace{-2.5cm} -\left(1-z_{mn}^{t}\right)\overline{\phi}_{mn} \le \phi_{mn}^{-,t} \le 0 , \hspace{-1.5cm} & mn \in \mathcal{C} & \label{OEGF_compminus} \\
			&\hspace{-3.5cm} -\left(1-z_{mn}^{t}\right)\overline{\phi}_{mn} \leq \phi_{mn}^{t} \leq z_{mn}^{t}\overline{\phi}_{mn}, \hspace{-1.5cm} & mn \in \mathcal{E} & \label{OEGF_compflow} \\
			&\hspace{-3.5cm} \left(\overline{\wp}_{n} - \underline{\wp}_{m} \overline{\gamma}_{mn}\right)\left(z_{mn}^{t} - 1\right) + \wp_{n}^{t} \le \wp_{m}^{t} \overline{\gamma}_{mn}, \hspace{-1cm} & mn \in \mathcal{E} & \label{OEGF_compboost1} \\
			&\hspace{-3.5cm} \left(\overline{\wp}_{m}\underline{\gamma}_{mn} - \underline{\wp}_{n} \right)\left(z_{mn}^{t} - 1\right) + \wp_{m}^{t}\underline{\gamma}_{mn} \le \wp_{n}^{t}, \hspace{-1cm} & mn \in \mathcal{E} & \label{OEGF_compboost2} \\
			&\hspace{-2.5cm} \wp_{n}^{t} - \wp_{m}^{t} \leq z_{mn}^{t}\left( \overline{\wp}_{n} - \underline{\wp}_{m} \right), \hspace{-1.5cm} & mn \in \mathcal{E} & \label{OEGF_compnoboost1} \\
			&\hspace{-2.5cm} \wp_{m}^{t} - \wp_{n}^{t} \leq z_{mn}^{t}\left(\overline{\wp}_{m} - \underline{\wp}_{n} \right), \hspace{-1.5cm} & mn \in \mathcal{E} & \label{OEGF_compnoboost2} \\
			&\hspace{-1.5cm} z_{mn}^{t} \in \left\lbrace 0 , 1 \right\rbrace, & \hspace{-3cm} mn \in \mathcal{E} & \label{OEGF_binary}
		\end{align}
	\end{subequations}
	for $t \in \mathcal{T}=\{t_{0},t_{0}+\Delta t,\ldots,t_{0} + \Delta t \left( H - 1 \right) \}$, where $\Delta \tau = 3600 \Delta t$, and $\mathcal{G} = \mathcal{G}^{\rm g} \cup \mathcal{G}^{\rm ng}$. The cost functions of electrical generators in the objective function in \eqref{OEGF_objective} are assumed to be quadratic, and of the form $f_{gi}\left(p_{gi}^{t}\right)=c_{2,gi}\left(p_{gi}^{t}\Delta t \right)^2 + c_{1,gi}\left(p_{gi}^{t}\Delta t \right) + c_{0,gi} \Delta t$. The cost functions of the gas supplies are assumed to be linear, and of the form $f_{sm}\left(\phi_{sm}^{t}\right) = c_{sm}\phi_{sm}^{t} \Delta \tau$. Electrical network constraints are delineated by \cref{OEGF_Pminmax,OEGF_Pramp,OEGF_MWlimits,OEGF_anglediff,OEGF_KCLp,OEGF_Pij,OEGF_Pji}, where \cref{OEGF_Pminmax} and \cref{OEGF_Pramp} capture the limits on active power generation and the generator ramp rates, respectively, whereas \eqref{OEGF_MWlimits} and \eqref{OEGF_anglediff} capture the limits on branch active power and angle difference, respectively. KCL is represented by \cref{OEGF_KCLp} and the strengthened DC power flow constraints are delineated by \cref{OEGF_Pij} and \cref{OEGF_Pji}, where $g_{ij}=\Re\left\{Y_{ij}^*/T_{ij}\right\}$, $b_{ij}=\Im\left\{Y_{ij}^*/T_{ij}\right\}$, $g_{ji}=\Re\left\{Y_{ji}^* / T_{ji}^*\right\}$, and $b_{ji}=\Im\left\{Y_{ji}^* / T_{ji}^*\right\}$. The strengthened DC OPF formulation is intended to approximate transmission line losses by replacing the $\cos\left( \theta_{ij} \right)$ terms in the original AC constraints by their second-order Maclaurin series; i.e., $\cos\left( \theta_{ij} \right) \approx 1 - 0.5\theta_{ij}^2 $, and is therefore a better approximation compared to its vanilla DC OPF counterpart. Gas system constraints are delineated by \cref{OEGF_Fminmax,OEGF_PressureMinMax,OEGF_Nodal,OEGF_gasflow_pipes,OEGF_averageflow,OEGF_averagepressure,OEGF_linepack1,OEGF_continuity,OEGF_FlowPipeMinMax,OEGF_compflowplusminus,OEGF_compgastrap,OEGF_compplus,OEGF_compminus,OEGF_compflow,OEGF_compboost1,OEGF_compboost2,OEGF_compnoboost1,OEGF_compnoboost2,OEGF_binary}, where \cref{OEGF_Fminmax} delineates capacity limits of gas supplies, \eqref{OEGF_PressureMinMax} captures the nodal pressure limits, \eqref{OEGF_Nodal} is the gas flow nodal balance, \cref{OEGF_gasflow_pipes,OEGF_averageflow,OEGF_averagepressure,OEGF_linepack1,OEGF_continuity,OEGF_FlowPipeMinMax} describe the \emph{quasi-dynamic} behavior of gas flow in a pipe, and \cref{OEGF_compflowplusminus,OEGF_compgastrap,OEGF_compplus,OEGF_compminus,OEGF_compflow,OEGF_compboost1,OEGF_compboost2,OEGF_compnoboost1,OEGF_compnoboost2,OEGF_binary} describe the operation of non-pipe elements, i.e., compressors and pressure regulators. In more detail, \eqref{OEGF_gasflow_pipes} is the discretized equation of \emph{motion} along the full length of the pipe \cite{Osiadacz1987_Simulationofgasnetworks}, where
	\begin{align*}
		\Phi_{mn} = \frac{\pi^2 D_{mn}^5}{16 \rho^2 Z_{mn} R T L_{mn} f_{mn}},
	\end{align*}
	and $f_{mn} = 4\left(20.621D_{mn}^{1/6} \right)^{-2}$ defines the \emph{Weymouth} friction factor \cite{Menon2005_PipelineHydrolics}. In this work, the compressibility factor is defined as
	\begin{multline*}
		\hspace{-0.3cm} Z_{mn} \hspace{-0.1cm} = \hspace{-0.1cm} \left( 1 + \frac{49.9511 \left( 10^{1.785 S}\right) \left(\tilde{\wp}_{m} + \tilde{\wp}_{n} - \frac{\tilde{\wp}_{m}\tilde{\wp}_{n}}{\tilde{\wp}_{m} + \tilde{\wp}_{n}} \right) }{\frac{3\left(1.8 T \right)^{3.825}}{2} }\right)^{-1} \hspace{-0.5cm},
	\end{multline*}
	where $\tilde{\wp}_{m} = 0.5 \left(\underline{\wp}_{m} + \overline{\wp}_{m} \right) $, as opposed to the gross simplification of assuming a constant compressibility factor across the whole network. The average flow and pressure across a pipe are captured by \eqref{OEGF_averageflow} and \eqref{OEGF_averagepressure}, respectively. Constraint \eqref{OEGF_linepack1} is the \emph{linepack} equation with
	\begin{align*}
		\Psi_{mn} = \frac{\pi D_{mn}^2 L_{mn}}{4\rho Z_{mn} R T},
	\end{align*} 
	and \eqref{OEGF_continuity} is the discretized \emph{continuity} equation over the full length of the pipe. The limits on the gas flowing in a pipeline are captured by \eqref{OEGF_FlowPipeMinMax}. The operation of bidirectional compressors and pressure regulators is captured by \cref{OEGF_compplus,OEGF_compminus,OEGF_compflow,OEGF_compboost1,OEGF_compboost2,OEGF_compnoboost1,OEGF_compnoboost2}, where $z_{mn}^{t} = 1$ if the flow is in the direction of pressure boost (drop) for a compressor (pressure regulator), and $z_{mn}^{t} = 0$ otherwise. Note that $\underline{\gamma}_{mn} \geq 1$ for a compressor and $\overline{\gamma}_{mn} \leq 1$ for a pressure regulator. Similarly, $\overline{\gamma}_{mn} > 1$ for a compressor and $0 < \underline{\gamma}_{mn} < 1$ for a pressure regulator. The gas trapped in a compressor ($\phi_{mn}^{{\rm c},t}$) when it is boosting pressure is captured by \eqref{OEGF_compgastrap}, where parameter $\nu_{mn}$ encapsulates the percentage of $\phi_{mn}^{+,t}$ that is trapped in the compressor. It is identified in \cite{BorrazSanchez2009_ImprovingPipelinesCyclic} that the gas trapped by a compressor typically ranges between 3\% to 5\% of the gas flowing across it. The two systems are coupled by \eqref{OEGF_gasflow_GFG}, which assumes a linear relationship between the power output of a GPG and its input gas consumption.
	
	The nonconvexity of the problem stems from constraints \cref{OEGF_Pij,OEGF_Pji}, \cref{OEGF_gasflow_pipes,OEGF_averagepressure}, and the disjoint sets in \cref{OEGF_compplus,OEGF_compminus,OEGF_compflow,OEGF_compboost1,OEGF_compboost2,OEGF_compnoboost1,OEGF_compnoboost2}.
	
	Problem~\ref{OEGF} can be written in the general form 
	\begin{subequations}\label{OEGF_Generalproblem}
		\begin{align}
			\underset {\substack{x}} 
			{\mbox{ minimize}} \quad f_{0}\left(x\right) & & \label{OEGF_Generalproblem_objective}\\
			\text{ subject to} \quad h_{i}\left(x\right) & = 0, & & i=1,\ldots,p_{\rm e} \label{OEGF_Generalproblem_eq_elec}\\
			h_{i}\left(x\right) & = 0, & & i=p_{\rm e}+1,\ldots,p \label{OEGF_Generalproblem_eq_gas}\\
			a_{i}^{T}x & \leq b_{i}, & & i=1,\ldots,o \label{OEGF_Generalproblem_linear} \\
			x_{i} & \in \left\lbrace 0,1 \right\rbrace, & & i=1,\ldots,r \label{OEGF_Generalproblem_integrality}
		\end{align}
	\end{subequations}
	where $x \in \reals^{n}$, $f_{0}$,$h_{1},\ldots,h_{p}: \reals^{n} \rightarrow \reals$, $a_{1},\ldots,a_{o} \in \reals^{n}$, $b_{1},\ldots,b_{o} \in \reals$, and $r < n$.\footnote{Recall that an affine equality constraint of the form $a_{i}^{T}x = b_{i}$ can be rewritten as $\left\lbrace a_{i}^{T}x \leq b_{i}\right\rbrace \cap \left\lbrace - a_{i}^{T}x \leq -b_{i} \right\rbrace $.} Function $f_{0}\left(x\right)$ represents the cost functions in \eqref{OEGF_objective}, whereas \eqref{OEGF_Generalproblem_eq_elec} represent \cref{OEGF_Pij} and \cref{OEGF_Pji}, and \eqref{OEGF_Generalproblem_eq_gas} represent \cref{OEGF_gasflow_pipes} and \cref{OEGF_averagepressure}. Equality constraints \cref{OEGF_Pij}, \cref{OEGF_Pji}, \cref{OEGF_gasflow_pipes}, and \cref{OEGF_averagepressure} are nonlinear and therefore nonconvex. The linear constraints in \eqref{OEGF_Generalproblem_linear} represent \cref{OEGF_Pminmax,OEGF_Pramp,OEGF_MWlimits,OEGF_anglediff,OEGF_KCLp}, \cref{OEGF_gasflow_GFG,OEGF_Fminmax,OEGF_PressureMinMax,OEGF_Nodal}, \cref{OEGF_averageflow}, and \cref{OEGF_linepack1,OEGF_continuity,OEGF_FlowPipeMinMax,OEGF_compflowplusminus,OEGF_compgastrap,OEGF_compplus,OEGF_compminus,OEGF_compflow,OEGF_compboost1,OEGF_compboost2,OEGF_compnoboost1,OEGF_compnoboost2}, whereas \eqref{OEGF_Generalproblem_integrality} represents integrality constraints \eqref{OEGF_binary}. Problem~\ref{OEGF} is an MINLP problem with a nonconvex continuous relaxation,\footnote{A continuous relaxation is obtained by relaxing the integrality constraints $z \in \left\lbrace 0,1 \right\rbrace$ into box constraints $z \in \left[ 0,1 \right] $.} which, compared to convex MINLP problems, is more challenging to solve using current state-of-the-art MINLP technology.
	This work therefore introduces (i) an iterative MILP-based method that leverages the computational efficiency of MILP solvers and the accuracy of IPM, and (ii) a fast iterative LP-based method that can be used in situations where only LP solvers are available. The optimality of the solution is measured by comparing it to the one obtained from the MICP relaxation described in the next section.

\section{Mixed-integer convex relaxation}\label{Sec_MICP}
	An MICP relaxation of Problem~\ref{OEGF} can be obtained by convexifying constraints \cref{OEGF_Pij,OEGF_Pji}, \cref{OEGF_gasflow_pipes}, and \cref{OEGF_averagepressure}. Quadratic equality constraints \cref{OEGF_Pij} and \cref{OEGF_Pji} can be straightforwardly convexified by relaxing them into inequality constraints of the form
	\begin{align}
		p_{ij}^{t} \ge g_{ij} \hspace{0.5cm} & \hspace{-0.5cm} 0.5(\theta_{ij}^{t})^2 + b_{ij}\theta_{ij}^{t}, \hspace{-1.5cm} & ij \in \mathcal{L} & \label{OEGF_MISOCP_Pij} \\
		p_{ji}^{t} \ge g_{ji} \hspace{0.5cm} & \hspace{-0.5cm} 0.5(\theta_{ij}^{t})^2 - b_{ji}\theta_{ij}^{t}, \hspace{-1.5cm} & ji \in \mathcal{L}&^{\rm t}. \hspace{-0.1cm} \label{OEGF_MISOCP_Pji}
	\end{align}
	Constraints \cref{OEGF_gasflow_pipes} can be equivalently rewritten as
	\begin{align}
		\left( \phi_{mn}^{t}\right)^2 = & \Phi_{mn} \left( 2z_{mn}^{t}-1 \right) \left((\wp_{m}^{t})^2 - (\wp_{n}^{t})^2 \right), \hspace{-0.2cm} & mn \in \mathcal{P} & \label{OEGF_gasflow_pipes_alt} \\
		-\overline{\phi}_{mn} & \left( 1 - z_{mn}^{t}\right) \le \phi_{mn} \le z_{mn}^{t} \overline{\phi}_{mn}, & mn \in \mathcal{P} & \label{OEGF_MISOCP_FlowPipeMinMax} \\
		\wp_{m}^{t} & - \wp_{n}^{t} \le z_{mn}^{t}\left( \overline{\wp}_{m} - \underline{\wp}_{n} \right) , & mn \in \mathcal{P} & \label{OEGF_MISOCP_PressureMax} \\
		\wp_{m}^{t} & - \wp_{n}^{t} \ge \left( 1 - z_{mn}^{t}\right) \left( \underline{\wp}_{m} - \overline{\wp}_{n}\right) , & mn \in \mathcal{P} & \label{OEGF_MISOCP_PressureMin} \\
		& z_{mn}^{t} \in \left\lbrace 0 , 1 \right\rbrace, & \hspace{-3cm} mn \in \mathcal{P} & \label{OEGF_MISOCP_binary}
	\end{align}
	where $z_{mn}^{t}$ is a binary variable that takes the value of 1 when the gas is flowing from node $m$ to node $n$ and the value of 0 when the gas is flowing from node $n$ to node $m$. Next, \eqref{OEGF_gasflow_pipes_alt} can be transformed into an SOC constraint by introducing two new variables, $\xi_{mn}^{t}$ and $\zeta_{mn}^{t}$, and constraints of the form
	\begin{align}
		\left( \phi_{mn}^{t}\right)^2 = & \Phi_{mn} \left((\xi_{mn}^{t})^2 - (\zeta_{mn}^{t})^2 \right), & mn \in \mathcal{P} & \label{OEGF_MISOCP_gasflow_pipes}\\
		\xi_{mn}^{t} - \wp_{m}^{t} & \le \left( 1 - z_{mn}^{t}\right) \left( \overline{\wp}_{n} - \underline{\wp}_{m} \right), & mn \in \mathcal{P} & \label{OEGF_MISOCP_xim_max} \\
		\xi_{mn}^{t} - \wp_{m}^{t} & \ge \left( 1 - z_{mn}^{t}\right) \left( \underline{\wp}_{n} - \overline{\wp}_{m} \right), & mn \in \mathcal{P} & \label{OEGF_MISOCP_xim_min} \\
		\xi_{mn}^{t} - \wp_{n}^{t} & \le z_{mn}^{t} \left( \overline{\wp}_{m} - \underline{\wp}_{n} \right), & mn \in \mathcal{P} & \label{OEGF_MISOCP_xin_max} \\
		\xi_{mn}^{t} - \wp_{n}^{t} & \ge z_{mn}^{t} \left( \underline{\wp}_{m} - \overline{\wp}_{n} \right), & mn \in \mathcal{P} & \label{OEGF_MISOCP_xin_min} \\
		\zeta_{mn}^{t} - \wp_{n}^{t} & \le \left( 1 - z_{mn}^{t}\right) \left( \overline{\wp}_{m} - \underline{\wp}_{n} \right), & mn \in \mathcal{P} & \label{OEGF_MISOCP_zetan_max} \\
		\zeta_{mn}^{t} - \wp_{n}^{t} & \ge \left( 1 - z_{mn}^{t}\right) \left( \underline{\wp}_{m} - \overline{\wp}_{n} \right), & mn \in \mathcal{P} & \label{OEGF_MISOCP_zetan_min} \\
		\zeta_{mn}^{t} - \wp_{m}^{t} & \le z_{mn}^{t} \left( \overline{\wp}_{n} - \underline{\wp}_{m} \right), & mn \in \mathcal{P} & \label{OEGF_MISOCP_zetam_max} \\
		\zeta_{mn}^{t} - \wp_{m}^{t} & \ge z_{mn}^{t} \left( \underline{\wp}_{n} - \overline{\wp}_{m} \right). & mn \in \mathcal{P} & \label{OEGF_MISOCP_zetam_min}
	\end{align}
	An MISOCP relaxation of \eqref{OEGF_gasflow_pipes} can now be obtained by relaxing \eqref{OEGF_MISOCP_gasflow_pipes} into the convex SOC constraint
	\begin{align}
		\left( \xi_{mn}^{t} \right)^2 \ge & \left( \zeta_{mn}^{t} \right)^2 + \left( \phi_{mn}^{t}\right)^2/\Phi_{mn} . & mn \in \mathcal{P} & \label{OEGF_MISOCP_gasflow_pipes_r}
	\end{align}

	More interestingly, and perhaps less obvious at first glance, because the nodal pressures are strictly positive, constraint \cref{OEGF_averagepressure} becomes convex when it is relaxed into an inequality constraint of the form
	\begin{align}
		& \tilde{\wp}_{mn}^{t} \ge \frac{2}{3} \left(\wp_{m}^{t} + \wp_{n}^{t} - \frac{\wp_{m}^{t}\wp_{n}^{t}}{\wp_{m}^{t} + \wp_{n}^{t}} \right). & mn \in \mathcal{P} & \label{OEGF_averagepressure_r}
	\end{align}
	\begin{theorem}\label{Averagepressure_convex}
		The function $f(x,y) = - xy/(x+y)$ with $\dom f = \reals_{++}^{2} = \left\lbrace (x,y) \in \reals^{2}| x > 0, y > 0 \right\rbrace $ is a convex function, which entails that its epigraph $\epi f = \left\{ (x,y,z) | (x,y) \in \dom f, f(x,y) \le z \right\}$ is a convex set.
	\end{theorem}
	\begin{proof}
		The proof proceeds by showing that the Hessian of $f(x,y)$ is positive semidefinite, i.e.,
		\begin{align}
			\hspace{-0.3cm} \nabla^{2}f(x,y) =
			\begin{bmatrix}
				2y^{2}/(x+y)^{3} & -2xy/(x+y)^{3} \\
				-2xy/(x+y)^{3} & 2x^{2}/(x+y)^{3}
			\end{bmatrix} \succeq 0.
		\end{align}
		Since $x > 0$ and $y > 0$, one can show that $z^{T}\nabla^{2}f(x,y)z \ge 0$ for all $z \in \reals^{2}$. In more detail, the above reduces to proving that $z_{1}^2y^2 + z_{2}^2x^2 - 2z_{1}z_{2}xy \ge 0$, which can be straightforwardly verified since $x^2 + y^2 -2xy \ge 0$ and $z_{1}^2 + z_{2}^2 - 2z_{1}z_{2} \ge 0$. This completes the proof.
	\end{proof}
	An MICP relaxation of Problem~\ref{OEGF} can now be formally written as
	\begin{align}
		\hspace{-1.25cm} \underset{\substack{p_{gi}^{t},p_{ij}^{t},\theta_{i}^{t},\wp_{m}^{t}\\\tilde{\wp}_{mn}^{t},\phi_{mn}^{t},\ell_{mn}^{t},\phi_{mn}^{{\rm in},t}\\\phi_{nm}^{{\rm out},t},z_{mn}^{t},\phi_{gm}^{t},\phi_{sm}^{t}\\\phi_{nm}^{{\rm c},t},\phi_{nm}^{+,t},\phi_{nm}^{-,t},\zeta_{mn}^{t},\xi_{mn}^{t}}} 
		{\mbox{ minimize}} & \sum_{t \in \mathcal{T}} \Bigg( \sum_{gi \in \mathcal{G}} \hspace{-0cm} f_{gi}\left(p_{gi}^{t}\right)\Bigg. + \hspace{-0.2cm} \hspace{-0cm} & \hspace{-0cm} & \nonumber
	\end{align}
	\vspace{-1.0cm}
	\begin{subequations}\label{OEGF_MISOCP}
		\begin{align}
			& \hspace{3cm} \Bigg. \sum_{sm\in \mathcal{S}} f_{sm}\left(\phi_{sm}^{t}\right)\Bigg) \hspace{-3cm} & & \label{OEGF_MISOCP_objective} \\
			& \hspace{-0.75cm} \text{subject to \cref{OEGF_Pminmax,OEGF_Pramp,OEGF_MWlimits,OEGF_anglediff,OEGF_KCLp}, \cref{OEGF_gasflow_GFG,OEGF_Fminmax,OEGF_PressureMinMax,OEGF_Nodal},  \cref{OEGF_linepack1,OEGF_continuity,OEGF_FlowPipeMinMax,OEGF_compflowplusminus,OEGF_compgastrap,OEGF_compplus,OEGF_compminus,OEGF_compflow,OEGF_compboost1,OEGF_compboost2,OEGF_compnoboost1,OEGF_compnoboost2}}, \hspace{-2cm} & & \nonumber \\ 
			& \hspace{0.75cm} \text{\cref{OEGF_MISOCP_Pij,OEGF_MISOCP_Pji}, \cref{OEGF_MISOCP_FlowPipeMinMax,OEGF_MISOCP_PressureMax,OEGF_MISOCP_PressureMin,OEGF_MISOCP_binary}, \cref{OEGF_MISOCP_xim_max,OEGF_MISOCP_xim_min,OEGF_MISOCP_xin_max,OEGF_MISOCP_xin_min,OEGF_MISOCP_zetan_max,OEGF_MISOCP_zetan_min,OEGF_MISOCP_zetam_max,OEGF_MISOCP_zetam_min,OEGF_MISOCP_gasflow_pipes_r,OEGF_averagepressure_r}}. \hspace{-2cm} & & \label{OEGF_MISOCP_common}
		\end{align}
	\end{subequations}
	Problem~\ref{OEGF_MISOCP} is a tractable \emph{convex} MINLP problem\footnote{An MINLP problem is called ``convex MINLP'' if it becomes a convex NLP when the integrality constraints are relaxed.} that can be solved using the (iterative) \emph{outer approximation} (OA) \cite{Duran1986_OuterApproximation}. Alternatively, if constraint~\eqref{OEGF_averagepressure_r} is replaced by a tight polyhedral envelope using a sufficiently large number of halfspaces, Problem~\ref{OEGF_MISOCP} can be transformed into an MISOCP that can be directly handled by powerful MISOCP solvers such as Gurobi \cite{Gurobi2019}. The construction of such a tight polyhedral envelope is discussed in the next section. Numerical evaluation in Section~\ref{Sec_NumericalEval} will show that, although tractable compared to solving Problem~\ref{OEGF} directly using NLBB solvers, the MICP problem in \eqref{OEGF_MISOCP} converges to \emph{infeasible} solutions to the original problem in \eqref{OEGF}. Nonetheless, this MICP relaxation provides a good-quality \emph{lower bound} for assessing the optimality of solutions to Problem~\ref{OEGF} from methods such as NLBB solvers and the ones proposed in this work.
	
\section{Polyhedral envelopes}\label{Sec_PolyEnv}

	As a first step towards formulating a computationally efficient LP approximation of Problem~\ref{OEGF}, the below discourse introduces polyhedral envelopes for the nonlinear terms in \cref{OEGF_Pij,OEGF_Pji}, \cref{OEGF_gasflow_pipes,OEGF_averagepressure}, and a PWL approximation of the quadratic cost functions in \eqref{OEGF_objective}.

\subsection{Polyhedral envelope of $x^{2}$}
	
	A polyhedral envelope of a nonconvex set of the form $\mathcal{V}= \left\{x^{2} | x \in \left[\underline{x},\overline{x}\right]\right\}$, can be obtained by, (i), finding the supporting hyperplanes passing through points $\left( \underline{x},\underline{x}^2 \right) $ and $\left( \overline{x},\overline{x}^2 \right) $, and (ii), finding the supporting hyperplanes obtained by \emph{outer} PWL approximations of the set. In more detail, $l$ points $x_1,\ldots,x_l$ are selected in the interval $\left[\underline{x},\overline{x}\right]$, which allows adding $(l-1) + 1$ halfspaces of the form
	\[$$
	\begin{subnumcases}{\label{convV} \hspace{-0mm} {\rm conv}\left(\mathcal{V}\right)= \!}
		\! v \geq \left(2 x_{h}\right)x-x_{h}^{2}, \quad h=\left\{1,\ldots,l\right\}, \label{convV_v1}\\
		\! v \leq \left(\overline{x}+\underline{x}\right)x-\overline{x} \underline{x} . \label{convV_v2}
	\end{subnumcases}
	$$ \nonumber\]

\subsection{Polyhedral envelope of $x \left| x \right|$}	
	
	\begin{figure}[t!]
		\centering{
			\includegraphics[width=\columnwidth] {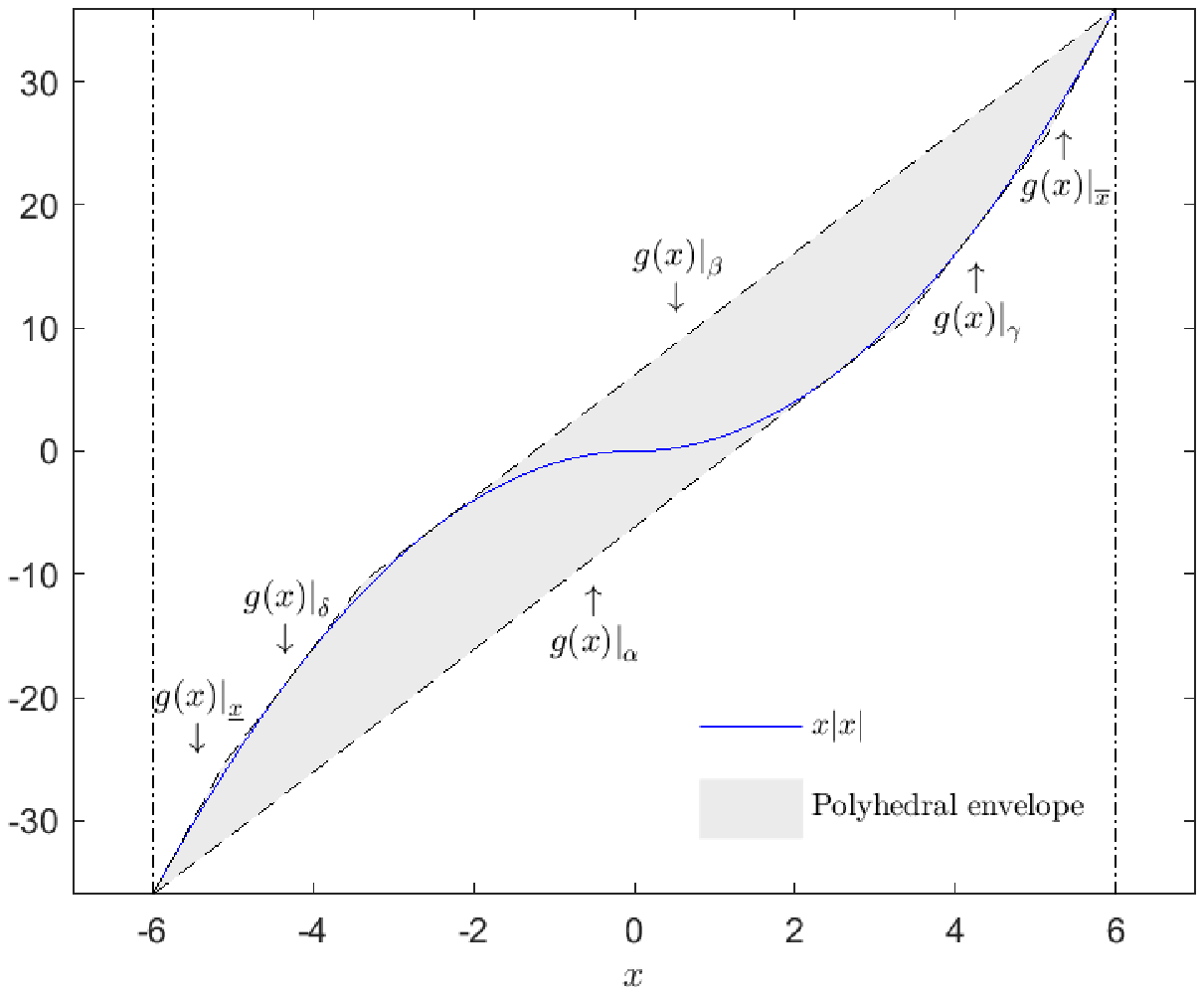}}
		\caption{Polyhedral envelope of $\mathcal{\mathcal{U}}= \left\{x \left| x \right| | x \in \left[\underline{x},\overline{x}\right]\right\}$.}
		\label{fig_xabsx}
	\end{figure}	
	A polyhedral envelope of a nonconvex set of the form $\mathcal{\mathcal{U}}= \left\{x \left| x \right| | x \in \left[\underline{x},\overline{x}\right]\right\}$ is obtained by determining the parameters of the four supporting hyperplanes passing through points $\left( \underline{x},\underline{x} \left|\underline{x} \right| \right) $ and $\left( \overline{x},\overline{x}\left|\overline{x} \right| \right) $. This can be done by separately solving two nonlinear equations in one dimension for each gas pipe. The first supporting hyperplane passing through $\underline{x}$ has the form $\left.g(x)\right|_{\underline{x}} = \underline{x} \left|\underline{x} \right| + 2\left| \underline{x}\right|\left(x - \underline{x} \right)$. The second supporting hyperplane passing through $\underline{x}$ can be obtained by solving the one-dimensional equation $\left.g(z)\right|_{\underline{x}} = z\left| z\right| + 2\left| z\right| \left(\underline{x}- z \right) = \underline{x} \left|\underline{x} \right|$, and	whose solution is denoted by $z^{\star} = \alpha = \underline{x}(1-\sqrt{2})$. Analogously, the first supporting hyperplane passing through $\overline{x}$ has the form $\left.g(x)\right|_{\overline{x}} = \overline{x} \left|\overline{x} \right| + 2\left| \overline{x}\right| \left(x - \overline{x} \right)$. The second supporting hyperplane passing through $\overline{x}$ can be obtained by solving $\left.g(z)\right|_{\overline{x}} = z\left| z\right| + 2\left| z\right| \left(\overline{x}- z \right) = \overline{x} \left|\overline{x} \right|$, and whose solution is denoted by $z^{\star} = \beta = \overline{x}(1-\sqrt{2})$. More supporting hyperplanes can be obtained in intervals $[\underline{x}, \beta]$ and $[\alpha, \overline{x}]$ to obtained a tighter polyhedral envelope but in this work only $\left.g(x)\right|_{\gamma}$ and $\left.g(x)\right|_{\delta}$ are added, where $\gamma$ and $\delta$ are obtained from the intersections of $\left.g(x)\right|_{\alpha}$ and $\left.g(x)\right|_{\overline{x}}$, and $\left.g(x)\right|_{\beta}$ and $\left.g(x)\right|_{\underline{x}}$, respectively. As a result, the polyhedral envelope of the nonconvex set $\mathcal{U}$ can be written as
	\begin{subnumcases}{\label{convU} \hspace{0mm} {\rm conv}\left(\mathcal{U}\right) = \!}
		\! u \geq x\left( 2 - \sqrt{8} \right)\underline{x} - \underline{x}^2\left( 3 - \sqrt{8} \right), \label{convU_u1} \\
		\! u \geq 2x\overline{x} - \overline{x}^2, \label{convU_u2} \\
		\! u \geq \gamma\left| \gamma\right| + 2\left| \gamma\right| \left(x- \gamma \right), \label{convU_u3} \\
		\! u \leq x\left(\sqrt{8} - 2\right)\overline{x} + \overline{x}^2\left(3 - \sqrt{8}\right), \label{convU_l1} \\
		\! u \leq -2x\underline{x} + \underline{x}^2, \label{convU_l2} \\
		\! u \leq \delta\left| \delta\right| + 2\left| \delta\right| \left(x - \delta \right). \label{convU_l3} 
	\end{subnumcases} 
	An illustration of \eqref{convU} is shown in Fig.~\ref{fig_xabsx}.\footnote{This formulation is valid for $\underline{x} < 0$ and $\overline{x} > 0$.}
	
\subsection{Polyhedral envelope of the average pressure}
	
	\begin{figure}[t]
		\centering{
			\includegraphics[width=\columnwidth] {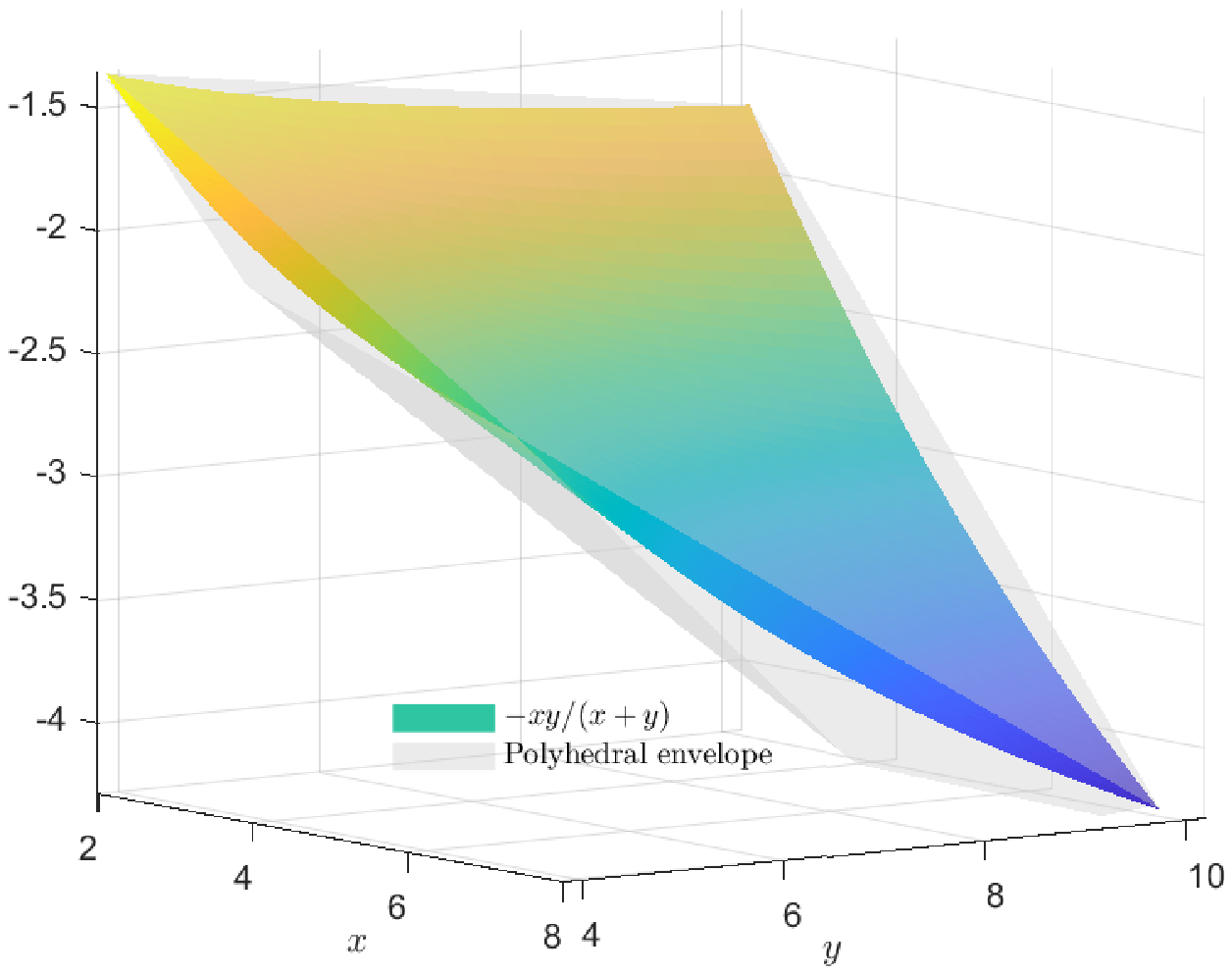}}
		\caption{Polyhedral envelope of $\mathcal{W}= \left\{ f(x,y) = -xy/(x+y) \right.$ $ \left. | \left(x,y\right) \in \left[\underline{x},\overline{x}\right] \times\left[\underline{y},\overline{y}\right]\right\}$.}
		\label{fig_xyoverxplusy}
	\end{figure}
	Following Theorem~\ref{Averagepressure_convex}, a tight polyhedral envelope of a nonconvex set of the form $\mathcal{W}= \left\{ f(x,y) \right. $ $\left. = -xy/(x+y) | \left(x,y\right) \in \left[\underline{x},\overline{x}\right] \times\left[\underline{y},\overline{y}\right]\right\}$ can be obtained from the halfspaces associated with a sufficiently large number of uniformly spaced points in the 2-D grid defined by $\dom f$. However, since the aim of this construction is computational efficiency and not tightness, only the four extreme points in the 2-D grid defined by $\dom f$ are chosen for the construction of the polyhedral envelope. These four extreme points are $p1 = \left(\underline{x},\underline{y}, f\left(\underline{x},\underline{y}\right)\right) $, $p2 = \left(\overline{x},\underline{y}, f\left(\overline{x},\underline{y}\right)\right) $, $p3 = \left(\underline{x},\overline{y}, f\left(\underline{x},\overline{y}\right)\right) $, and $p4 = \left(\overline{x},\overline{y}, f\left(\overline{x},\overline{y}\right)\right) $. In more detail, the first supporting hyperplane in the 3-D Euclidean space, which can be obtained from points $p1$, $p2$, and $p3$, is written as $g\left( x,y\right) = f\left(\underline{x},\underline{y}\right) - \frac{a_{x}}{a_{z}}\left( x - \underline{x} \right) - \frac{a_{y}}{a_{z}}\left( y - \underline{y} \right)$, where $a = \left(p2-p1 \right) \times \left(p3-p1 \right) $ is the normal vector. Similarly, the second supporting hyperplane can be obtained from points $p2$, $p3$, and $p4$, and can be written as $g\left( x,y\right) = f\left(\overline{x},\overline{y}\right) - \frac{b_{x}}{b_{z}}\left( x - \overline{x} \right) - \frac{b_{y}}{b_{z}}\left( y - \overline{y} \right)$, where $b = \left(p2-p4 \right) \times \left(p3-p4 \right) $ is the normal vector. The remaining four supporting hyperplanes can be obtained from the first-order Taylor series approximation of $f(x,y)$ at the four extreme points $p1$, $p2$, $p3$, and $p4$. The polyhedral envelope of $\mathcal{W}$ can now be written as
	\begin{flalign*}
		{\rm conv}\left(\mathcal{W}\right) = &&
	\end{flalign*}
	\begin{subnumcases}{\label{convW} \hspace{-0mm} \!}
		\! w \geq f(\overline{x},\overline{y}) + \frac{\partial f(\overline{x},\overline{y})}{\partial x} \left( x - \overline{x} \right) + \frac{\partial f(\overline{x},\overline{y})}{\partial y} \left( y - \overline{y} \right), \hspace{-0.5cm} & \label{convW_u1} \\
		\! w \geq f(\underline{x},\underline{y}) + \frac{\partial f(\underline{x},\underline{y})}{\partial x} \left( x - \underline{x} \right) + \frac{\partial f(\underline{x},\underline{y})}{\partial y} \left( y - \underline{y} \right), \hspace{-0.5cm} & \label{convW_u2} \\
		\! w \geq f(\overline{x},\underline{y}) + \frac{\partial f(\overline{x},\underline{y})}{\partial x} \left( x - \overline{x} \right) + \frac{\partial f(\overline{x},\underline{y})}{\partial y} \left( y - \underline{y} \right), \hspace{-0.5cm} & \label{convW_u3} \\
		\! w \geq f(\underline{x},\overline{y}) + \frac{\partial f(\underline{x},\overline{y})}{\partial x} \left( x - \underline{x} \right) + \frac{\partial f(\underline{x},\overline{y})}{\partial y} \left( y - \overline{y} \right), \hspace{-0.5cm} & \label{convW_u4} \\
		\! w \leq f(\underline{x},\underline{y}) - \frac{a_{x}}{a_{z}}\left( x - \underline{x} \right) - \frac{a_{y}}{a_{z}}\left( y - \underline{y} \right), & \label{convW_l1} \\
		\! w \leq f(\overline{x},\overline{y}) - \frac{b_{x}}{b_{z}}\left( x - \overline{x} \right) - \frac{b_{y}}{b_{z}}\left( y - \overline{y} \right). & \label{convW_l2}
	\end{subnumcases}
	An illustration of \eqref{convW} is shown in Fig.~\ref{fig_xyoverxplusy}.

\subsection{PWL approximation of a quadratic cost function}
	
	The last step is to substitute the quadratic terms in $f_{gi}\left( p_{gi}^{t} \right)$ by corresponding variables and rotated SOC constraints which can then be tightly approximated by a lifted polyhedron as in \cite{Mhanna2016_TightLPfortheOPF}. This construction requires far less inequality constraints than standard piecewise linear (PWL) approximations for the same accuracy. Specifically, $\left|\mathcal{G}\right| \times \left|\mathcal{T}\right|$ variables $p_{g}^{t}$ and associated linear constraints of the form $p_{g}^{t}=\Delta t \sqrt{c_{2,gi}}p_{gi}^{t}$, for all $ gi \in \mathcal{G}$, $t \in \mathcal{T}$, are introduced along with $N \times \left|\mathcal{T}\right|$ variables $\delta_{n}^{t}$ and constraints of the form $\delta_{n}^{t} \geq ( p_{2n-1}^{t} )^2 + ( p_{2n}^{t} )^2 $, for which the lifted polyhedral construction is denoted as
	\begin{align}\label{RSOC_Obj}
		\hspace{-0.2cm} \mathcal{P}_{k}^{\rm R} \left( \left( p_{2n-1}^{t}\right)^2 + \left( p_{2n}^{t} \right)^2 \leq \delta_{n}^{t} \right), n \in \left\{1,\ldots,N\right\}, t \in \mathcal{T}
	\end{align}
	where $N = \left\lfloor \left|\mathcal{G}\right| /2\right\rfloor+\left\lceil \left| \mathcal{G} \right| /2 - \left\lfloor \left| \mathcal{G} \right| /2\right\rfloor \right\rceil$.

\subsection{Polyhedral relaxation of the MOEGF}
	
	In light of the above, the LP-relaxed MOEGF with the polyhedral envelopes and PWL objective function can be written as
	\begin{align}
		\hspace{-0.5cm} \underset{\substack{p_{gi}^{t},\delta_{n}^{t},p_{ij}^{t},\theta_{i}^{t},v_{ij}^{t},\wp_{m}^{t}\\v_{m}^{t},\tilde{\wp}_{mn}^{t},w_{mn}^{t},\phi_{mn}^{t},u_{mn}^{t}\\\ell_{mn}^{t},\phi_{mn}^{{\rm in},t},\phi_{nm}^{{\rm out},t},z_{mn}^{t}\\\phi_{gm}^{t},\phi_{sm}^{t},\phi_{nm}^{{\rm c},t},\phi_{nm}^{+,t},\phi_{nm}^{-,t} }} 
		{\mbox{ minimize}} & \sum_{t \in \mathcal{T}} \Bigg( \sum_{n=1}^{N} \delta_{n}^{t} + \hspace{-0.0cm} \sum_{gi \in \mathcal{G}} \hspace{-0cm} c_{1,gi}\left(p_{gi}^{t}\Delta t \right) \Bigg. \hspace{-0.2cm} & \hspace{-0cm} & \nonumber
	\end{align}
	\vspace{-1.0cm}
	\begin{subequations}\label{OEGF_Poly}
		\begin{align}
			& \hspace{2cm} + c_{0,gi} \Delta t + \Bigg. \sum_{sm\in \mathcal{S}} f_{sm}\left(\phi_{sm}^{t}\right)\Bigg) \hspace{-3cm} & & \label{OEGF_Poly_objective} \\
			& \hspace{-0.75cm} \text{subject to \cref{OEGF_Pminmax,OEGF_Pramp,OEGF_MWlimits,OEGF_anglediff,OEGF_KCLp}, \cref{OEGF_gasflow_GFG,OEGF_Fminmax,OEGF_PressureMinMax,OEGF_Nodal}, \cref{OEGF_averageflow}, \cref{OEGF_linepack1,OEGF_continuity,OEGF_FlowPipeMinMax,OEGF_compflowplusminus,OEGF_compgastrap,OEGF_compplus,OEGF_compminus,OEGF_compflow,OEGF_compboost1,OEGF_compboost2,OEGF_compnoboost1,OEGF_compnoboost2}, \cref{RSOC_Obj}} \hspace{-2cm} & & \label{OEGF_Poly_common} \\ 
			p_{ij}^{t}&=g_{ij}0.5{\rm conv}\left( \left(\theta_{ij}^{t}\right)^2 \right) + b_{ij}\theta_{ij}^{t}, & \hspace{-0.0cm} ij \in \mathcal{L} & \label{OEGF_Poly_Pij} \\
			p_{ji}^{t}&= g_{ji}0.5{\rm conv}\left( \left(\theta_{ij}^{t}\right)^2 \right) - b_{ji}\theta_{ij}^{t}, & \hspace{-0.0cm} ji \in \mathcal{L}&^{\rm t} \label{OEGF_Poly_Pji} \\
			{\rm conv}&\left( \phi_{mn}^{t}\left|\phi_{mn}^{t} \right|\right) = \Phi_{mn} {\rm conv}\left( \left(\wp_{m}^{t}\right)^2 \right) \hspace{-0.5cm} & & \nonumber \\
			& \hspace{2.5cm} - {\rm conv}\left( \left(\wp_{n}^{t}\right)^2 \right), & mn \in \mathcal{P} & \label{OEGF_Poly_gasflow_pipes}\\
			& \hspace{-1cm} \tilde{\wp}_{mn}^{t} = \frac{2}{3} \left(\wp_{m}^{t} + \wp_{n}^{t} + {\rm conv}\left( \frac{-\wp_{m}^{t}\wp_{n}^{t}}{\wp_{m}^{t} + \wp_{n}^{t}} \right) \right), \hspace{-0.3cm} & mn \in \mathcal{P} & \label{OEGF_Poly_averagepressure} \\
			&\hspace{2cm} z_{mn}^{t} \in \left[ 0 , 1 \right] , & \hspace{-0cm} mn \in \mathcal{E} & \label{OEGF_Poly_nobinary}
		\end{align}
	\end{subequations}
	for $t \in \mathcal{T}$. These polyhedral envelopes depend only on the properties and topology of the network and can thus be formed once, in a single computationally-cheap preprocessing stage. The solution of Problem \ref{OEGF_Poly} can be used as a warm start for the two iterative algorithms discussed in the next two sections.

\section{Iterative MILP-based method}\label{Sec_SMILP}

	\begin{algorithm}[t!]
		\caption{Iterative MILP-based method}
		\begin{adjustwidth}{-1em}{}
			\begin{algorithmic}[1]
				\scriptsize
				\STATE \parbox[t]{\dimexpr\linewidth-0.75cm}{\textbf{Initialization:}
					Set $k=1$, $\overline{K} = 40$, $\epsilon \in \left[ 10^{-10}, 10^{-4} \right] $, $\sigma >> 0$ , $\overline{\alpha} > \alpha_{mn}^{t,(1)} > 0$ for all $ mn \in \mathcal{P}$ and $t \in \mathcal{T}$, $\gamma > 0$, and $\mathcal{I}_{i} = \left\lbrace \emptyset \right\rbrace$ for all $ i \in \left\lbrace 1,\ldots,p_{\rm e} \right\rbrace $. Obtain the initial point $x^{(1)}$ by either solving Problem~\ref{OEGF_Poly} (warm start) or by choosing $x^{(1)}$ randomly such that $x^{(1)}$ is in the domain of the all constraint functions of Problem~\ref{OEGF_Generalproblem}. The point $x^{(1)}$ need not be feasible.\strut}
				\algrule
				\normalsize
				\STATEx {\qquad \qquad \qquad \qquad \ \underline{Phase I (LP)}}
				\scriptsize
				\WHILE {${\rm max}\left( \left( \left| h_{i}\left( x^{(k)}\right) \right| \right)_{i=1,\ldots,p}\right) > \epsilon$}
					\STATE \parbox[t]{\dimexpr\linewidth-0.75cm}{
						\begin{subequations}\label{OEGF_TaylorI}
							\begin{align}
								& \hspace{-0.6cm} x^{(k+1)} = & & \nonumber \\
								& \hspace{-0.6cm} \underset {\substack{x,r_{ij}^{t}}} 
								{\argmin} \ \eqref{OEGF_Poly_objective} + \sum_{t \in \mathcal{T}} \left( \sum_{ij \in \mathcal{L}} \sigma r_{ij}^{t} + \hspace{-0.1cm} \sum_{mn \in \mathcal{P}} \hspace{-0.0cm} \alpha_{mn}^{t,(k)} \left|\phi_{mn}^{t} - \phi_{mn}^{t,(k)} \right| \right) \hspace{-5cm} & & \label{OEGF_Taylor_objectiveI}\\
								& \text{subject to } a_{i}^{T}x \leq b_{i}, \hspace{-0.0cm} & & i=1,\ldots,o \label{OEGF_Taylor_linear} \\
								& \hspace{-0.5cm} h_{i}(x^{(k)}) + \nabla h_{i}(x^{(k)})^{T}(x - x^{(k)}) = r_{i}, \hspace{-0.0cm} & & i = 1,\ldots,p_{\rm e} \label{OEGF_Taylor_eq_elec_eq}\\
								& \hspace{-0.5cm} h_{i}(x^{(\kappa)}) + \nabla h_{i}(x^{(\kappa)})^{T}(x - x^{(\kappa)}) \leq 0, \hspace{-0.0cm} & & i = 1,\ldots,p_{\rm e}, \kappa \in \mathcal{I}_{i} \label{OEGF_Taylor_eq_elec_ineq}\\
								& \hspace{-0.5cm} r_{i} \ge 0, \hspace{-0.0cm} & & i = 1,\ldots,p_{\rm e} \label{OEGF_Taylor_eq_elec_ri}\\
								& \hspace{-0.5cm} h_{i}(x^{(k)}) + \nabla h_{i}(x^{(k)})^{T}(x - x^{(k)}) = 0, \hspace{-0.31cm} & & i=p_{\rm e}+1,\ldots,p \label{OEGF_Taylor_eq_gas}\\
								& \hspace{1cm} x_{i} \in \left[ 0,1 \right] , \hspace{-0.31cm} & & i=1,\ldots,r. \label{OEGF_Taylor_01_box}
							\end{align}
						\end{subequations} \vspace{-0.2cm} \strut.}
					\FOR {$i = 1,\ldots,p_{\rm e}$}
						\IF {$h_{i}\left( x^{(k+1)}\right) < 0$}
							\STATE { $\mathcal{I}_{i} := \mathcal{I}_{i} \cup \{k\} $.}
						\ENDIF
					\ENDFOR 
					\FOR {$i = p_{\rm e}+1,\ldots,p$}
						\IF {$ \left| h_{i}\left( x^{(k+1)}\right) \right| \geq \left| h_{i}\left( x^{(k)}\right) \right| \textbf{and} \left| h_{i}\left( x^{(k+1)}\right) \right| > \epsilon $}
							\STATE $\alpha_{mn}^{t,(k+1)} := \min\left( \overline{\alpha}, \gamma \alpha_{mn}^{t,(k)} \right) $.
						\ENDIF
					\ENDFOR 
					\STATE {$k \leftarrow k + 1 $.}
				\ENDWHILE
				\algrule
				\normalsize
				\STATEx {\qquad \qquad \qquad \qquad \ \underline{Phase II (MILP)}}
				\scriptsize
				\WHILE {${\rm max}\left( \left( \left| h_{i}\left( x^{(k)}\right) \right| \right)_{i=1,\ldots,p}\right) > \epsilon$}
					\STATE \parbox[t]{\dimexpr\linewidth-0.75cm}{
						\begin{subequations}\label{OEGF_TaylorI_MILP}
							\begin{align}
								& \hspace{-1cm} x^{(k+1)} = \underset {\substack{x,r_{ij}^{t}}} 
								{\argmin} \ \eqref{OEGF_Taylor_objectiveI} \hspace{-5cm} & & \label{OEGF_Taylor_objectiveI_MILP}\\
								& \hspace{-1cm} \text{subject to \cref{OEGF_Taylor_linear,OEGF_Taylor_eq_elec_eq,OEGF_Taylor_eq_elec_ineq,OEGF_Taylor_eq_elec_ri,OEGF_Taylor_eq_gas}} & & \label{OEGF_Taylor_Common} \\
								& \hspace{0cm} x_{i} \in \left\lbrace 0,1 \right\rbrace, \hspace{-0.31cm} & & i=1,\ldots,r. \label{OEGF_Taylor_01}
							\end{align}
						\end{subequations} \vspace{-0.2cm} \strut.}
						\STATE {Same as lines 4 to 13.}
						\STATE {$k \leftarrow k + 1 $.}
				\ENDWHILE
			\end{algorithmic} 
		\end{adjustwidth}
		\label{Alg_SMILP}
	\end{algorithm}
	The proposed MILP-based method for directly solving Problem~\ref{OEGF} is described in Algorithm~\ref{Alg_SMILP}, which is divided into two phases. Phase I consists of a sequence of carefully  coordinated LP problems for solving the continuous relaxation of Problem~\ref{OEGF}, i.e., Problem~\ref{OEGF} with \eqref{OEGF_Poly_nobinary} instead of \eqref{OEGF_binary}. The solution from Phase I is then used as a warm start for Phase II which consists of a sequence of carefully coordinated MILP problems that converges to a feasible solution to Problem~\ref{OEGF}. Phase I is akin to SLP which was originally introduced in \cite{Cheney1959_NewtonsmethodandTchebycheffApprox,Kelley1960_CuttingPlaneMethod}. SLP, which is also known at the \emph{cutting plane method}, consists of solving the original NLP problem by solving a series of LP problems generated by approximating all the nonlinear constraints by their first-order Taylor series expansion around the current point $x^{(k)}$, where $k$ is the iteration number. The iterative procedure in Phase I can either be initialized from the solution of Problem \ref{OEGF_Poly} or from a random point that lies in the domain of the all constraint functions of Problem~\ref{OEGF}. The purpose of Problem \ref{OEGF_Poly} is two fold. First, it provides a valid starting point for Algorithm~\ref{Alg_SMILP} when no prior information is available. Second, the superior quality of this starting point, compared to a random point, manifests in an improvement in overall convergence. The small computational overhead of Problem \ref{OEGF_Poly} is far outweighed by the improvement in overall convergence, as will be demonstrated in Section~\ref{Sec_NumericalEval}.
		
	The convergence of Phase I is enabled by two key features. The first consists of a finite set of closed halfspaces, delineated by \eqref{OEGF_Taylor_eq_elec_ineq}, implemented in conjunction with a finite set of hyperplanes of the form \eqref{OEGF_Taylor_eq_elec_eq}. However, instead of directly using supporting hyperplanes of the form $\left\lbrace x | h_{i}(x^{(k)}) + \nabla h_{i}(x^{(k)})^{T}(x - x^{(k)}) = 0, i=1,\ldots,p_{\rm e} \right\rbrace $ (to the nonconvex set of the form $\mathcal{V}$), which might lead to infeasible LP problems, Algorithm~\ref{Alg_SMILP} introduces non-negative slack variables $r_{ij}^{t}$ whose purpose is to prevent infeasible LP problems, especially during the first few iterations of the algorithm if it is initialized from a poor-quality starting point (random $x^{(1)}$). These variables are then minimized in the objective by assigning a relatively large value for parameter $\sigma$. This construction can be viewed as an iteratively refined polyhedral outer approximation of the nonconvex sets of the form $\mathcal{V}= \left\{x^{2} | x \in \left[\underline{x},\overline{x}\right]\right\}$, but with the use of both supporting hyperplanes \emph{and} halfspaces. Set $\mathcal{I}_{i}$ registers all previous iteration numbers that qualify as supporting halfspaces for $\mathcal{V}= \left\{x^{2} | x \in \left[\underline{x},\overline{x}\right]\right\}$. The condition on line 5 is satisfied when constraints \eqref{OEGF_Pij}-\eqref{OEGF_Pji} are violated at iteration $k$, in which case the supporting halfspace of the form \eqref{OEGF_Taylor_eq_elec_ineq} is added to Problem \ref{OEGF_TaylorI} at $k+1$. The second key feature is the term $\alpha_{mn}^{t,(k)} \left|\phi_{mn}^{t} - \phi_{mn}^{t,(k)} \right|$ in \eqref{OEGF_Taylor_objectiveI}, which controls the step size of the LP approximation of \eqref{OEGF_gasflow_pipes}. The step size parameter $\alpha_{mn}^{t,(k)}$ is automatically tuned on lines 9 to 13, predicated on the condition on line 10 which is true when the violations of constraints \eqref{OEGF_gasflow_pipes} and \eqref{OEGF_averagepressure} do not decrease as the algorithm iterates.
	
	Phase I terminates when the constraint with the largest violation, $C_{\rm max}^{k} = {\rm max}\left( \left( \left| h_{i}\left( x^{(k)}\right) \right| \right)_{i=1,\ldots,p}\right) $, does not exceed a certain tolerance $\epsilon$.\footnote{Feasibility here is measured with respect to Problem~\ref{OEGF} and not to the original problem with AC power flow constraints (which is not shown here).} The solution obtained at the termination of Phase I is likely to be non-integral, i.e., some of the $z_{mn}^{t}$ variables may not be strictly 0 or 1. This solution is therefore used as a warm start for Phase II which is iterative MILP-based algorithm that closely resembles Phase I but with the integrality constraints \eqref{OEGF_Taylor_01} instead of the relaxed ones in \eqref{OEGF_Taylor_01_box}. Phase II, and therefore Algorithm~\ref{Alg_SMILP}, terminates under the same conditions in Phase I but now with a guarantee that the integrality constraints are satisfied. The next section introduces a fast iterative LP-based method that can be used as an alternative to Algorithm~\ref{Alg_SMILP} in situations where only LP solvers are available.

\section{Iterative LP-based method}\label{Sec_SLPH}
	
	\begin{algorithm}[t!]
		\caption{Iterative LP-based mehod}
		\begin{adjustwidth}{-1em}{}
			\begin{algorithmic}[1]
				\scriptsize
				\STATE \parbox[t]{\dimexpr\linewidth-0.75cm}{\textbf{Initialization:}
					Set $k=1$, $\overline{K} = 40$, $k_{2} = 0$, $K_{f} = 2$, $\epsilon \in \left[ 10^{-10}, 10^{-4} \right] $, $\sigma >> 0$, $\beta > 0$, $\overline{\alpha} > \alpha_{mn}^{t,(1)} > 0$ for all $ mn \in \mathcal{P}$ and $t \in \mathcal{T}$, $\gamma > 0$, and $\mathcal{I}_{i} = \left\lbrace \emptyset \right\rbrace$ for all $ i \in \left\lbrace 1,\ldots,p_{\rm e} \right\rbrace $. Obtain the initial point $x^{(1)}$ by either solving Problem~\ref{OEGF_Poly} (warm start) or by choosing $x^{(1)}$ randomly such that $x^{(1)}$ is in the domain of the all constraint functions of Problem~\ref{OEGF_Generalproblem}. The point $x^{(1)}$ need not be feasible.\strut}
				\algrule
				\normalsize
				\STATEx {\qquad \qquad \qquad \qquad \ \underline{Phase I (LP)}}
				\scriptsize
				\STATE {Same as lines 2 to 15 of Algorithm~\ref{Alg_SMILP}.}
				\algrule
				\normalsize
				\STATEx {\qquad \qquad \qquad \qquad \ \underline{Phase II (LP)}}
				\scriptsize
				\FOR {$mn \in \mathcal{E}, t \in \mathcal{T}$}
					\IF {$\phi_{mn}^{t,(k)} \leq 0$}
						\STATE {$I_{mn}^{t} \leftarrow 0 $.}
					\ELSIF {$\phi_{mn}^{t,(k)} > 0$}
						\STATE {$I_{mn}^{t} \leftarrow 1 $.}
					\ENDIF
				\ENDFOR
				\WHILE {$ k \le \overline{K}$}
					\STATE \parbox[t]{\dimexpr\linewidth-0.75cm}{
						\begin{subequations}\label{OEGF_TaylorII}
							\begin{align}
								& \hspace{-1cm} x^{k+1} = \underset {\substack{x},r_{ij}^{t}} 
								{\argmin} \ \eqref{OEGF_Taylor_objectiveI} + \sum_{mn \in \mathcal{E}, t \in \mathcal{T}} \beta \left|z_{mn}^{t} - I_{mn}^{t} \right| \hspace{0cm} & & \label{OEGF_Taylor_objectiveII}\\
								& \hspace{-0cm} \text{subject to \cref{OEGF_Taylor_linear,OEGF_Taylor_eq_elec_eq,OEGF_Taylor_eq_elec_ineq,OEGF_Taylor_eq_elec_ri,OEGF_Taylor_eq_gas,OEGF_Taylor_01_box}}. & &
							\end{align}
						\end{subequations} \vspace{-0.2cm} \strut} 
					\IF{${\rm max}\left( \left( \left| h_{i}\left( x^{(k)}\right) \right| \right)_{i=1,\ldots,p}\right) \le \epsilon$ \textbf{and} all $z_{mn}^{t,(k+1)} \in \left\lbrace 0,1 \right\rbrace $}
						\STATE {\textbf{Break}.}
					\ELSE
						\IF {$ \text{mod}\left(k_{2},K_{f} \right) = 0 $ }
							\FOR {$mn \in \mathcal{E}, t \in \mathcal{T}$}
								\IF {$I_{mn}^{t} == 0$ \textbf{and} $z_{mn}^{t,(k+1)} > \epsilon$}
									\STATE $I_{mn}^{t} \leftarrow 1 $.
								\ELSIF {$I_{mn}^{t} == 1$ \textbf{and} $z_{mn}^{t,(k+1)} < 1 - \epsilon$}
									\STATE $I_{mn}^{t} \leftarrow 0 $.
								\ENDIF
							\ENDFOR
						\ENDIF
					\ENDIF	
					\STATE {$k \leftarrow k + 1 $.}
					\STATE {$k_{2} \leftarrow k_{2} + 1 $.}
				\ENDWHILE
			\end{algorithmic} 
		\end{adjustwidth}
		\label{Alg_SLPH}
	\end{algorithm}
	
	The LP-based heuristic, described in Algorithm~\ref{Alg_SLPH}, starts exactly like Algorithm~\ref{Alg_SMILP} by solving the continuous relaxation of Problem~\ref{OEGF} in Phase I, whose solution is then passed on to the LP-based heuristic in Phase II. The solution obtained at the termination of Phase I is likely to be non-integral, i.e., some of the $z_{mn}^{t}$ variables may not be strictly 0 or 1. Phase II therefore consists of a heuristic that ensures that the solution is integral and feasible within a tolerance $\epsilon$. In a nutshell, instead of a branch-and-bound algorithm, the heuristic in Phase II implements an elaborate \emph{steering} procedure controlled by the term $\beta \left|z_{mn}^{t} - I_{mn}^{t} \right|$ in \eqref{OEGF_Taylor_objectiveII}. More specifically, lines 3 to 8 set \emph{parameter} $I_{mn}^{t}$ for both compressors and pressure regulators to either 0 or 1 based on the direction of flow $\phi_{mn}^{t,(k)}$ from the solution of Phase I. These values of $I_{mn}^{t}$ are then passed on to Problem \ref{OEGF_TaylorII}. The purpose of lines 15 to 23 is to handle any remaining non-integral values of $z_{mn}^{t}$ after solving \eqref{OEGF_TaylorII}. There are two reasons a variable $z_{mn}^{t}$ can remain non-integral after solving \eqref{OEGF_TaylorII} at iteration $k$. The first pertains to the potential for decreasing the objective function value, and the second is related to feasibility. For instance, if $I_{mn}^{t}$ was set to 1 but \eqref{OEGF_TaylorII} resulted in $z_{mn}^{t,(k+1)} = 0.85$, the heuristic uses this as a clue to \emph{steer} $z_{mn}^{t,(k+2)}$ towards 0. In other words, what this means is that had $z_{mn}^{t,(k+1)}$ been forced to 1, the problem would have either converged to a solution with a larger objective function value, or to an infeasible point. Analogously, if $I_{mn}^{t}$ was set to 0 but \eqref{OEGF_TaylorII} resulted in $z_{mn}^{t,(k+1)} = 0.25$ for instance, the heuristic uses this as a clue to steer $z_{mn}^{t,(k+2)}$ towards 1. In other words, what this means is that had $z_{mn}^{t,(k+1)}$ been forced to 0, the problem would have either converged to a solution with a larger objective function value, or to an infeasible point. This procedure is only invoked every $K_{f}$ iterations, as opposed to at every iteration, in an effort to prevent potential oscillatory behavior. In summary, the procedure in Phase II does not force the binary variables to either 0 or 1, but instead uses the term $\beta \left|z_{mn}^{t} -I_{mn}^{t} \right|$ to steer these variables to take binary values predicated on a careful tuning of parameter $\beta$. Therefore, the value of $\beta$ is chosen to strike a good tradeoff between objective value minimization and infeasibility handling. Phase II is terminated when $C_{\rm max}^{k} = {\rm max}\left( \left( \left| h_{i}\left( x^{(k)}\right) \right| \right)_{i=1,\ldots,p}\right) \leq \epsilon$ and when all the $z_{mn}^{t}$ variables are integral.
	
	The procedure in Phase II, although technically a heuristic, is successful for this type of problem mainly because the direction of gas flow in compressors and regulators is predominantly determined by pipeline constraints \eqref{OEGF_gasflow_pipes} and the location of the gas demands. The heuristic in Algorithm 1 therefore exploits this property of the problem to recover an integral solution. Just like other local MINLP solvers such as Juniper and KNITRO \cite{KNITRO}, Algorithm~\ref{Alg_SLPH} cannot guarantee global optimality.
	
\section{Numerical evaluation}\label{Sec_NumericalEval} 
	
	\begin{table}[t!]
		\normalsize
		\centering
		\caption{Elements of the test cases.}
		\resizebox{\linewidth}{!}{%
			\begin{tabular}{ r c c c c c c c c c c}
				\hline
				Case	&	$\left| \mathcal{B} \right| $	&	$\left| \mathcal{L} \right| $	&	$\left| \mathcal{G}^{\rm ng} \right| $	&	$\left| \mathcal{G}^{\rm g} \right| $	&	$\left| \mathcal{N} \right| $	&	$\left| \mathcal{P} \right| $	&	$\left| \mathcal{C} \right| $	&	$\left| \mathcal{R} \right| $	&	$\left| \mathcal{S} \right| $	\\\hline
				A	&	5	&	6	&	3	&	2	&	7	&	4	&	2	&	0	&	2	\\\hline
				B	&	14	&	20	&	3	&	2	&	25	&	24	&	6	&	0	&	6	\\\hline
				C	&	72	&	122	&	14	&	17	&	47	&	39	&	10	&	7	&	5	\\\hline
		\end{tabular}}
		\label{table_Test_cases}
	\end{table}
	In this experimental setup, Julia v1.4.0 \cite{Bezanson2017_Julia} is used as a programming language along with JuMP v0.21.1 \cite{DunningHuchetteLubin2017_JuMP} as a frontend modeling language for the optimization problems. Gurobi v9.0.2 \cite{Gurobi2019} with the Barrier method (without a crossover strategy) is used to solve all the LP problems (including Problem~\ref{OEGF_Poly}). The MICP optimization problems in \eqref{OEGF_MISOCP} are also solved using Gurobi v9.0.2 but now with the \emph{linearized outer approximation} algorithm. The continuous relaxation of Problem~\ref{OEGF}, which is a nonconvex NLP problem, is solved using IPOPT v3.12.11 \cite{IPOPT} and linear solver MA57 \cite{MA57}. The original MINLP problem in \eqref{OEGF} is solved using the NLBB solver Juniper \cite{Juniper} with IPOPT v3.12.11 \cite{IPOPT} and linear solver MA57 \cite{MA57} for the NLP subproblems, and Gurobi v9.0.2 in the \emph{feasibility pump} heuristic at the root node. All simulations are conducted on a computing platform with an Intel Core i7-6820HK CPU at 2.7GHz, 64-bit operating system, and 32GB RAM. Five test systems are considered for the numerical evaluation of Algorithm~\ref{Alg_SMILP} with hourly ($\Delta t = 1$) demand and linepack data over $H = 24$ hours. Test case A consists of a 7-node gas system connected to a 5-bus electrical system, test case B consists of the Belgian gas system \cite{Zhang2018_MILPforIEGS} connected to the IEEE 14-bus electrical system \cite{Castillo2016_SLPforOPF}, and test cases C1, C2, and C3 consist of the real electricity and gas system of the state of Victoria, Australia for low (23/11/2019), medium (21/08/2019), and high (09/08/2019) demand days, respectively. Data for the actual gas network of the state of Victoria, the Victorian Declared Transmission System (DTS), was developed from \emph{scratch} with the help of industry support within the Future Fuels CRC project \cite{FFCRC}. The electrical network data for the state of Victoria is obtained from \cite{Xenophon2018_OpenGridModel} and updated to reflect the generation mix and network augmentations of 2019. A summary description of the three IEGS test cases is shown in Table~\ref{table_Test_cases}, and detailed data can be found in \cite{IEGS_data}.
	
	In Problem~\ref{OEGF_Poly}, functions $f_{gi}\left( p_{gi}^{t}\right) $ in \eqref{OEGF_Poly_objective} are approximated by a lifted polyhedron with accuracy of $1.15 \times 10^{-9}$ (see \cite{Mhanna2016_TightLPfortheOPF}), whereas $l=10$ is chosen for the \emph{outer} PWL approximation in the polyhedral envelopes of the quadratic terms in \cref{OEGF_Poly_Pij,OEGF_Poly_Pji}, and \cref{OEGF_Poly_gasflow_pipes}. The angle limits are set to $\underline{\theta}_{ij} = -45^{\circ}$ and $\overline{\theta}_{ij} = 45^{\circ}$, and are considered generous as in practice $\theta_{ij} = \overline{\theta}_{ij} - \underline{\theta}_{ij}$ typically does not exceed $\pm 10^{\circ}$ \cite{Purchala2005_UsefulnessofDCPF}. Furthermore, in the aim of improving the numerical conditioning of the problem, all the units are nondimensionalized with a base apparent power of $\SI{100}{\mega\VA}$, base flow rate of $\SI{100}{\meter\cubed\per\second}$, and base pressure of $10^6$$\SI{}{\pascal}$. Algorithm~\ref{Alg_SMILP} is initialized with $\epsilon = 10^{-4}$, $\alpha_{mn}^{t,(1)} = 0.1$, $\gamma = 10$, $\overline{\alpha} = 1000$, and $\sigma = 0.1{\max} \left( \left( {\max}\left( c_{2,gi},c_{1,gi}\right)  \right) _{gi \in \mathcal{G}} \right)$ for all test systems. In Algorithm~\ref{Alg_SLPH}, $\beta = 0.1 {\max} \left( \left(c_{sm} \right) _{sm \in \mathcal{S}} \right) $.
	
	The below starts with a numerical assessment of Phase I of Algorithm~\ref{Alg_SMILP} in Section~\ref{Sec_PhaseI}, followed by an assessment of Algorithm~\ref{Alg_SMILP} and Algorithm~\ref{Alg_SLPH} in Sections~\ref{Sec_Alg_SMILP} and \ref{Sec_Alg_SLPH}. Two metrics are used for measuring optimality. The first is the \emph{optimality gap} (\%), defined as
	\begin{align*}
		Ogap= \left( f_{0}\left( x^{\star} \right) - f_{0}\left( x^{ \rm cvx} \right) \right) /f_{0}\left( x^{\star} \right) \times 100,
	\end{align*}
	where $x^{\star}$ is the solution obtained by a local solver such as IPOPT, Juniper or Algorithms~\ref{Alg_SMILP} and~\ref{Alg_SLPH}, and $x^{\rm cvx}$ is the solution obtained by an MICP relaxation such as the one in Problem~\ref{OEGF_MISOCP}.\footnote{Note that IPOPT, Juniper, and Algorithms~\ref{Alg_SMILP} and~\ref{Alg_SLPH} cannot guarantee global optimality in this case. IPOPT is a local primal-dual IPM that consists of a sequence of second-order approximations coordinated by a filter line search method whereas Phase I of Algorithms~\ref{Alg_SMILP} and~\ref{Alg_SLPH} consists of a sequence of first-order approximations coordinated by \eqref{OEGF_Taylor_eq_elec_eq}, \eqref{OEGF_Taylor_eq_elec_ineq}, and the second and third terms in \eqref{OEGF_Taylor_objectiveII}.} The second is the \emph{relative optimality gap} (\%), defined as
	\begin{align*}
		ROgap= \left( f_{0}\left( x^{\dagger} \right) - f_{0}\left( x^{\star} \right) \right) /f_{0}\left( x^{\dagger} \right) \times 100,
	\end{align*} 
	where $x^{\star}$ is the solution obtained by a local IPM solver (IPOPT in this case), and $x^{\rm \dagger}$ is the solution obtained by Phase I of Algorithm~\ref{Alg_SMILP}.

\subsection{Evaluation of Phase I of Algorithm~\ref{Alg_SMILP}}\label{Sec_PhaseI}
	
	\begin{table*}[t]
			\centering
			\caption{Optimality and feasibility of Phase I of Algorithm~\ref{Alg_SMILP} compared to MOEGFc (solved using IPOPT \cite{IPOPT}), and MICPc (solved using Gurobi \cite{Gurobi2019}).}
			\resizebox{\linewidth}{!}{
				\begin{tabular}{| c | c | c | c | c | c | c | c | c | c | c |}
					\cline{2-11}
					\multicolumn{1}{c|}{} & \multicolumn{2}{c}{\textbf{\rule{0pt}{7pt} MICPc} (Gurobi)} & \multicolumn{2}{|c|}{\textbf{\rule{0pt}{7pt} MOEGFc} (IPOPT)} & \multicolumn{6}{c|}{\textbf{Phase I} (Gurobi)} \\
					\cline{1-1} \cline{4-11}
					\multicolumn{1}{|c|}{\textbf{Test}} & \multicolumn{2}{c|}{} & \multicolumn{2}{c|}{Cold start} & \multicolumn{3}{c|}{Cold start} & \multicolumn{3}{c|}{Warm start (Problem~\ref{OEGF_Poly})} \\
					\cline{2-3} \cline{4-11}
					\rule{0pt}{10pt} \textbf{case} & Cost (\$) & $C_{\rm max}^{\rm cvx}$ & Cost (\$) & $Ogap$ (\%) & Cost (\$) & $ROgap$ (\%) & $C_{\rm mean}^{\dagger}$ & Cost (\$) & $ROgap$ (\%) & $C_{\rm mean}^{\dagger}$ \\\hline
					A	&	99946.5	&	394.67	&	99946.5	&	-1.26E-06	&	99946.5	&	-1.02E-06	&	1.86E-09	&	99946.5	&	-1.15E-06	&	2.16E-10	\\\hline
					B	&	93999.6	&	195.00	&	94000.4	&	8.87E-04	&	94003.9	&	-3.70E-03	&	9.73E-08	&	94000.4	&	-3.94E-05	&	1.67E-07	\\\hline
					C1	&	2975270.4	&	10.25	&	2975990.2	&	2.42E-02	&	2975988.5	&	5.73E-05	&	2.51E-09	&	2975988.2	&	6.79E-05	&	3.98E-10	\\\hline
					C2	&	6315991.7	&	8.96	&	6318017.0	&	3.21E-02	&	6318445.6	&	-6.78E-03	&	2.90E-08	&	6318097.4	&	-1.27E-03	&	8.97E-09	 \\\hline
					C3	&	11358689.3	&	9.96	&	11359946.4	&	1.11E-02	&	11360188.9	&	-2.14E-03	&	3.78E-08	&	11359700.5	&	2.16E-03	&	1.42E-08	 \\\hline																	
			\end{tabular}}
			\label{Table_Opt_PhaseI}
	\end{table*}
	\begin{table}[t]
			\normalsize
			\centering
			\caption{Computational effort of Phase I of Algorithm~\ref{Alg_SMILP} compared to MOEGFc (solved using IPOPT \cite{IPOPT}), and MICPc (solved using Gurobi \cite{Gurobi2019}). The numbers in parenthesis denote the number of iterations of Phase I.}
			\resizebox{\linewidth}{!}{%
				\begin{tabular}{| c | c | c | c | c |}
					\hline
					\multicolumn{1}{|c|}{Test} & \multicolumn{4}{c|}{CPU time (s)} \\
					\cline{2-5}
					\rule{0pt}{10pt} Case & MICPc & IPOPT & Phase I (Cold) & Phase I (Warm) \\\hline
					A	&	8.1	&	1.5	&	0.2 (3)	&	0.3 (5)	\\\hline
					B	&	2856.3	&	32.0	&	4.1 (13)	&	2.2 (7)	\\\hline
					C1	&	8873.6	&	937.0	&	12.7 (10)	&	14.8 (12)	\\\hline
					C2	&	10751.7	&	993.0	&	19.1 (11)	&	17.2 (11)	 \\\hline
					C3	&	12069.1	&	551.3	&	34.4 (11)	&	18.3 (10)	 \\\hline
			\end{tabular}}
			\label{Table_CPU_PhaseI}
	\end{table}
	This section assesses the optimality, feasibility, and computational efficiency of Phase I of Algorithm~\ref{Alg_SMILP} (which is the same as Phase I of Algorithm~\ref{Alg_SLPH}), IPOPT on the continuous relaxation of Problem~\ref{OEGF}, and the MICP relaxation in \eqref{OEGF_MISOCP} with \eqref{OEGF_Poly_nobinary} instead of \eqref{OEGF_binary}. For ease of exposition, these three approaches are called ``Phase I'', ``MOEGFc'', and ``MICPc'', respectively. Two starting point strategies are considered in the initialization of Phase I and IPOPT (MOEGFc). The first is a \emph{cold start} strategy that assumes $\theta_{ij}^{t,(1)} = 0.01$ for all $ij \in \mathcal{L}$, $\phi_{mn}^{t,(1)} = 0.1\overline{\phi}_{mn}$ for all $mn \in \mathcal{P}$, and $\wp_{m}^{t,(1)} = 0.5(\overline{\wp}_{m}+\underline{\wp}_{m})$ for all $m \in \mathcal{N}$, $t \in \mathcal{T}$. The \emph{warm start} strategy consists of using the solution of Problem~\ref{OEGF_Poly} in the initialization of Phase I. Table~\ref{Table_Opt_PhaseI} compares the optimality and feasibility of Phase I to those of IPOPT on the continuous relaxation of Problem~\ref{OEGF} (MOEGFc), and Table~\ref{Table_CPU_PhaseI} compares the computational effort of the three.\footnote{Warm starting IPOPT did not noticeably affect the quality or the computational effort of the solution. Those results are therefore not shown in Table~\ref{Table_CPU_PhaseI}.} Table~\ref{Table_Opt_PhaseI} shows the objective function values of MICPc, MOEGFc, and Phase I under cold start, and Phase I under warm start in columns 2, 4, 6, and 9, respectively. Table~\ref{Table_Opt_PhaseI} also shows the maximum constraint violation of the solution of Problem~\ref{OEGF_MISOCP} with \eqref{OEGF_Poly_nobinary} instead of \eqref{OEGF_binary} (i.e., MICPc) in column 3 under $C_{\rm max}^{\rm cvx}$. These large constraint violations confirm the infeasibility of the MICPc solutions. Nonetheless, the MICPc can still provide a good-quality lower bound on the solution of MOEGFc as shown in column 5. More specifically, it can be inferred that the MOEGFc solution to test Case A is optimal, and the solution to test case C3 is within $1.65 \times 10^{-2}$\% of the optimum. Additionally, since the largest $ROgap$ in columns 7 and 10 does not exceed $5 \times 10^{-3}$, and the largest constraint violation at the termination of Phase I is around $10^{-5}$ (i.e., $C_{\rm max}^{\dagger} \approx 10^{-5}$), it can be concluded that Phase I reaches near-optimal and feasible solutions to MOEGFc. The mean violations, $C_{\rm mean}^{\dagger} = {\rm mean}\left( \left( \left| h_{i}\left( x^{\dagger}\right) \right| \right)_{i=1,\ldots,p}\right)$, at the termination of Phase I are in the order of $10^{-8}$ on average for all test cases. Decreasing the constraint feasibility tolerance beyond $10^{-4}$ results in negligible change in the objective function and in the total linepack in all test cases.
	The evolution of the maximum constraint violation $C_{\rm max}^{k}$ for test case C3 is shown in Fig.~\ref{fig_Max_viol_SLP}.
	\begin{figure}[t]
		\centering{
			\psfrag{k}{\footnotesize $k$ \normalsize}
			\psfrag{C}{\footnotesize $C_{\rm max}^{k}$ \normalsize}
			\psfrag{Victorian IEGS}{\footnotesize Test case C3 \normalsize}
			\includegraphics[width=\columnwidth] {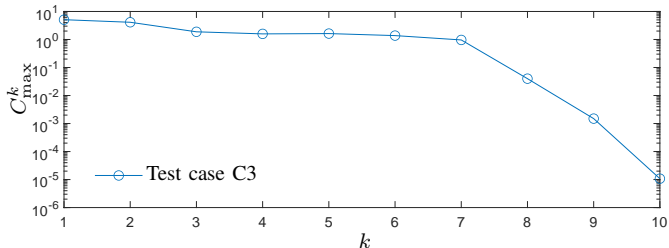}}
		\caption{Phase I's $C_{\rm max}^{k}$ vs $k$ for test case C3 under a warm start.}
		\label{fig_Max_viol_SLP}
	\end{figure}
	Furthermore, Table~\ref{Table_CPU_PhaseI} shows that Phase I solves MOEGFc in computation times that are at least one order of magnitude smaller than those of IPOPT. On the other hand, the computational effort of solving MICPc is disproportionately larger than those of IPOPT and Phase I due to the presence of binary variables (one for each pipeline as in \eqref{OEGF_MISOCP_binary}), which increases the complexity of the problem. Finally, although only two starting point strategies are shown in Tables \ref{Table_Opt_PhaseI} and \ref{Table_CPU_PhaseI}, the SLP algorithm underpinning Phase I is in fact robust to the choice of starting point $x^{(1)}$, owing to (i) the introduction slack variables to prevent infeasible problems, (ii) an automatic tuning of the convergence parameter $\alpha_{mn}^{t,(k)}$, and (iii) an iteratively refined polyhedral outer approximation of the nonconvex sets of the form $\mathcal{V}= \left\{x^{2} | x \in \left[\underline{x},\overline{x}\right]\right\}$ with the use of both supporting hyperplanes \emph{and} halfspaces. The computational advantages on the warm start (using the solution of Problem~\ref{OEGF_Poly}) become more prominent in the solution of Problem~\ref{OEGF}, as will be shown in the next two sections.

\subsection{Evaluation of Algorithm~\ref{Alg_SMILP}}\label{Sec_Alg_SMILP}
	
	\begin{table*}[t]
			\centering
			\caption{Optimality and feasibility of Algorithm~\ref{Alg_SMILP} compared to MOEGF (solved using Juniper \cite{Juniper}), and MICP (solved using Gurobi \cite{Gurobi2019}). ``NA'' designates instances that did not converge after one day of run time.}
			\resizebox{\linewidth}{!}{
				\begin{tabular}{| c | c | c | c | c | c | c | c | c | c | c |}
					\cline{2-11}
					\multicolumn{1}{c}{} & \multicolumn{2}{|c|}{\textbf{\rule{0pt}{7pt} MICP} (Gurobi)} & \multicolumn{2}{|c|}{\textbf{\rule{0pt}{7pt} MOEGF} (Juniper)} & \multicolumn{6}{c|}{\textbf{Algorithm~\ref{Alg_SMILP}} (Gurobi)} \\
					\cline{1-1} \cline{4-11}
					\multicolumn{1}{|c|}{\textbf{Test}} & \multicolumn{2}{c|}{} & \multicolumn{2}{c|}{Cold start} & \multicolumn{3}{c|}{Cold start} & \multicolumn{3}{c|}{Warm start (Problem~\ref{OEGF_Poly})} \\
					\cline{2-3} \cline{4-11}
					\rule{0pt}{10pt} \textbf{case} & Cost (\$) & $C_{\rm max}^{\rm cvx}$ & Cost (\$) & $Ogap$ (\%) & Cost (\$) & $Ogap$ (\%) & $C_{\rm mean}^{\dagger}$ & Cost (\$) & $Ogap$ (\%) & $C_{\rm mean}^{\dagger}$ \\\hline
					A	&	99946.5	&	400.66	&	99946.5	&	-1.71E-06	&	99946.7	&	2.06E-04	&	3.41E-08	&	99946.3	&	-2.47E-04	&	2.93E-08	\\\hline
					B	&	93999.8	&	189.17	&	NA	&	NA	&	94003.8	&	4.25E-03	&	2.60E-08	&	94001.6	&	1.88E-03	&	2.13E-08	\\\hline
					C1	&	2975274.4	&	10.72	&	NA	&	NA	&	2975884.3	&	2.05E-02	&	7.70E-08	&	2975975.6	&	2.36E-02	&	9.91E-09	\\\hline
					C2	&	6316166.5	&	8.16	&	NA	&	NA	&	6318099.4	&	3.06E-02	&	1.76E-08	&	6318009.4	&	2.92E-02	&	1.76E-08	 \\\hline
					C3	&	11358722.1	&	10.36	&	NA	&	NA	&	11360828.9	&	1.85E-02	&	5.72E-08	&	11360640.7	&	1.69E-02	&	3.08E-08	 \\\hline																	
			\end{tabular}}
			\label{Table_Opt_SMILP}
	\end{table*}
	\begin{table}[t]
			\normalsize
			\centering
			\caption[Performance Alg_SMILP]{Computational effort of Algorithm~\ref{Alg_SMILP} compared to Juniper \cite{Juniper}, and MICP (solved using Gurobi \cite{Gurobi2019}). The numbers in parenthesis denote the total number of iterations of  Algorithm~\ref{Alg_SMILP}. ``DNC'' designates instances that did not converge after one day of run time.\footnotemark}
			\resizebox{\linewidth}{!}{%
				\begin{tabular}{| c | c | c | c | c |}
					\hline
					\multicolumn{1}{|c|}{Test} & \multicolumn{4}{c|}{CPU time (s)} \\
					\cline{2-5}
					\rule{0pt}{10pt} Case & MICP & Juniper & Alg.~\ref{Alg_SMILP} (Cold) & Alg.~\ref{Alg_SMILP} (Warm) \\\hline
					A	&	9.6	&	2.9	&	3.3 (6)	&	1.6 (7)	\\\hline
					B	&	3986.2	&	DNC	&	14.4 (16)	&	11 (10)	\\\hline
					C1	&	12750.4	&	DNC	&	59.8 (24)	&	40.5 (17)	\\\hline
					C2	&	14867.6	&	DNC	&	77.8 (19)	&	65.3 (16)	 \\\hline
					C3	&	17205.7	&	DNC	&	536.7 (24)	&	305.6 (21)	 \\\hline					
			\end{tabular}}
			\label{Table_CPU_SMILP}
	\end{table}
	\footnotetext{The computation times of Algorithm~\ref{Alg_SMILP} with the warm start include the CPU times of Problem \ref{OEGF_Poly}.}
	This section assesses the optimality, feasibility, and computational efficiency of Algorithm~\ref{Alg_SMILP}, Juniper on Problem~\ref{OEGF}, and the MICP relaxation in \eqref{OEGF_MISOCP}. The starting point strategies adopted here are similar to the ones discussed in the previous section. The optimality and feasibility, and the computational effort of Algorithm~\ref{Alg_SMILP} compared to those of Juniper and the MICP relaxation in \eqref{OEGF_MISOCP} are shown in Tables~\ref{Table_Opt_SMILP} and~\ref{Table_CPU_SMILP}, respectively. It is evident from those two tables that the original MINLP problem in \eqref{OEGF} is extremely challenging to solve using an NLBB solver such as Juniper. In fact, Juniper did not converge (after one day of run time) on the practical-size systems in test cases B, C1, C2, C3. As a result, Table~\ref{Table_Opt_SMILP} only shows the optimality gap, as opposed to the relative optimality gap, of the solution of Algorithm~\ref{Alg_SMILP} as it is the only available feasible solution to Problem~\ref{OEGF} for test cases B, C1, C2, C3. The MICP relaxation in \eqref{OEGF_MISOCP} is also more challenging to solve due the additional number of binary variables associated with the non-pipe elements (constraint \eqref{OEGF_MISOCP_binary}). On test case C3, the MICP relaxation now takes more than 4 hours to converge. On the other hand, Algorithm~\ref{Alg_SMILP} with the warm start strategy takes 5 minutes. Not only Algorithm~\ref{Alg_SMILP} is substantially faster than the MICP relaxation in \eqref{OEGF_MISOCP}, it also converges to near-optimal and feasible solutions ($C_{\rm max}^{\dagger} \approx 10^{-5}$) with an average constraint violation in the order of $10^{-8}$. The solution to test case C3 is within $2.67 \times 10^{-2}$\% of the optimum. In comparison, the MICP relaxation converges to infeasible solutions (with large constraint violations $C_{\rm max}^{\rm cvx}$) in substantially longer computation times. The evolution of the maximum constraint violation $C_{\rm max}^{k}$ for test case C3 is shown in Fig.~\ref{fig_Max_viol_SMILP}.
	\begin{figure}[t]
		\centering{
			\psfrag{k}{\footnotesize $k$ \normalsize}
			\psfrag{C}{\footnotesize $C_{\rm max}^{k}$ \normalsize}
			\psfrag{Victorian IEGS}{\footnotesize Test case C3 \normalsize}
			\psfrag{Phase I}{\footnotesize Phase I \normalsize}
			\psfrag{Phase II}{\footnotesize Phase II \normalsize}
			\includegraphics[width=\columnwidth] {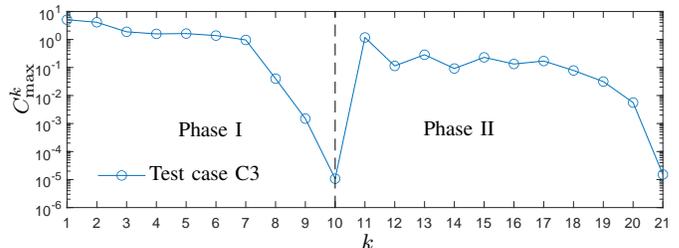}}
		\caption{Algorithm~\ref{Alg_SMILP}'s $C_{\rm max}^{k}$ vs $k$ for test case C3 under a warm start.}
		\label{fig_Max_viol_SMILP}
	\end{figure}
	Finally, the computational advantage of warm starting Algorithm~\ref{Alg_SMILP} from the solution of Problem \ref{OEGF_Poly} as compared to a cold start is more salient in this case, as the last column of Table~\ref{Table_CPU_SMILP} shows a noticeable improvement across all the test cases. In fact, Algorithm~\ref{Alg_SMILP} with the warm start strategy is at least one order of magnitude faster than Juniper and the MICP relaxtion in \eqref{OEGF_MISOCP}.

\subsection{Evaluation of Algorithm~\ref{Alg_SLPH}}\label{Sec_Alg_SLPH}

	\begin{table*}[t]
			\centering
			\caption{Optimality and feasibility of Algorithm~\ref{Alg_SLPH} compared to MOEGF (solved using Juniper \cite{Juniper}), and MICP (solved using Gurobi \cite{Gurobi2019}). ``NA'' designates instances that did not converge after one day of run time.}
			\resizebox{\linewidth}{!}{
				\begin{tabular}{| c | c | c | c | c | c | c | c | c | c | c |}
					\cline{2-11}
					\multicolumn{1}{c}{} & \multicolumn{2}{|c|}{\textbf{\rule{0pt}{7pt} MICP} (Gurobi)} & \multicolumn{2}{|c|}{\textbf{\rule{0pt}{7pt} MOEGF} (Juniper)} & \multicolumn{6}{c|}{\textbf{Algorithm~\ref{Alg_SLPH}} (Gurobi)} \\
					\cline{1-1} \cline{4-11}
					\multicolumn{1}{|c|}{\textbf{Test}} & \multicolumn{2}{c|}{} & \multicolumn{2}{c|}{Cold start} & \multicolumn{3}{c|}{Cold start} & \multicolumn{3}{c|}{Warm start (Problem~\ref{OEGF_Poly})} \\
					\cline{2-3} \cline{4-11}
					\rule{0pt}{10pt} \textbf{case} & Cost (\$) & $C_{\rm max}^{\rm cvx}$ & Cost (\$) & $Ogap$ (\%) & Cost (\$) & $Ogap$ (\%) & $C_{\rm mean}^{\dagger}$ & Cost (\$) & $Ogap$ (\%) & $C_{\rm mean}^{\dagger}$ \\\hline
					A	&	99946.5	&	400.66	&	99946.5	&	-1.71E-06	&	99948.8	&	2.30E-03	&	2.71E-07	&	99948.8	&	2.29E-03	&	2.40E-07	\\\hline
					B	&	93999.8	&	189.17	&	NA	&	NA	&	94004.7	&	5.22E-03	&	7.41E-08	&	94001.9	&	2.24E-03	&	2.23E-08	\\\hline
					C1	&	2975274.4	&	10.72	&	NA	&	NA	&	2979358.7	&	1.37E-01	&	1.63E-08	&	2979435.8	&	1.40E-01	&	4.20E-08	\\\hline
					C2	&	6316166.5	&	8.16	&	NA	&	NA	&	6322072.4	&	9.34E-02	&	1.50E-07	&	6321209.9	&	7.98E-02	&	4.31E-08	 \\\hline
					C3	&	11358722.1	&	10.36	&	NA	&	NA	&	11364530.9	&	5.11E-02	&	3.56E-08	&	11364465.6	&	5.05E-02	&	2.27E-08	 \\\hline																
			\end{tabular}}
			\label{Table_Opt_SLPH}
	\end{table*}
	\begin{table}[t]
			\normalsize
			\centering
			\caption[Performance Alg_SLPH]{Computational effort of Algorithm~\ref{Alg_SLPH} compared to Juniper \cite{Juniper}, and MICP (solved using Gurobi \cite{Gurobi2019}). The numbers in parenthesis denote the total number of iterations of Algorithm~\ref{Alg_SLPH}. ``DNC'' designates instances that did not converge after one day of run time.\footnotemark}
			\resizebox{\linewidth}{!}{%
				\begin{tabular}{| c | c | c | c | c |}
					\hline
					\multicolumn{1}{|c|}{Test} & \multicolumn{4}{c|}{CPU time (s)} \\
					\cline{2-5}
					\rule{0pt}{10pt} Case & MICP & Juniper & Alg.~\ref{Alg_SLPH} (Cold) & Alg.~\ref{Alg_SLPH} (Warm) \\\hline
					A	&	9.6	&	2.9	&	0.4 (6)	&	0.5 (8)	\\\hline
					B	&	3986.2	&	DNC	&	4.9 (16)	&	2.9 (10)	\\\hline
					C1	&	12750.4	&	DNC	&	23.3 (19)	&	21.1 (21)	\\\hline
					C2	&	14867.6	&	DNC	&	37.2 (17)	&	30.7 (17)	 \\\hline
					C3	&	17205.7	&	DNC	&	71.3 (20)	&	37.3 (18)	 \\\hline
			\end{tabular}}
			\label{Table_CPU_SLPH}
	\end{table}
	\footnotetext{The computation times of Algorithm~\ref{Alg_SLPH} with the warm start include the CPU times of Problem \ref{OEGF_Poly}.}
	The optimality and feasibility, and the computational effort of Algorithm~\ref{Alg_SLPH} compared to those of Juniper (on Problem~\ref{OEGF}) and the MICP relaxation in \eqref{OEGF_MISOCP} are shown in Tables~\ref{Table_Opt_SLPH} and~\ref{Table_CPU_SLPH}, respectively. In this case, the optimality gaps on the Victorian test cases C1, C2, and C3 are still less than 0.2\%, which corroborates the high quality of the solutions. The solutions are also feasible to within at most $1.8 \times 10^{-5}$ (i.e., $C_{\rm max}^{\dagger} \approx 1.8 \times 10^{-5}$) with an average constraint violation in the order of $10^{-8}$. The evolution of the maximum constraint violation $C_{\rm max}^{k}$ for test case C3 is shown in Fig.~\ref{fig_Max_viol_SLPH}.
	\begin{figure}[t]
		\centering{
			\psfrag{k}{\footnotesize $k$ \normalsize}
			\psfrag{C}{\footnotesize $C_{\rm max}^{k}$ \normalsize}
			\psfrag{Victorian IEGS}{\footnotesize Test case C3 \normalsize}
			\psfrag{Phase I}{\footnotesize Phase I \normalsize}
			\psfrag{Phase II}{\footnotesize Phase II \normalsize}
			\includegraphics[width=\columnwidth] {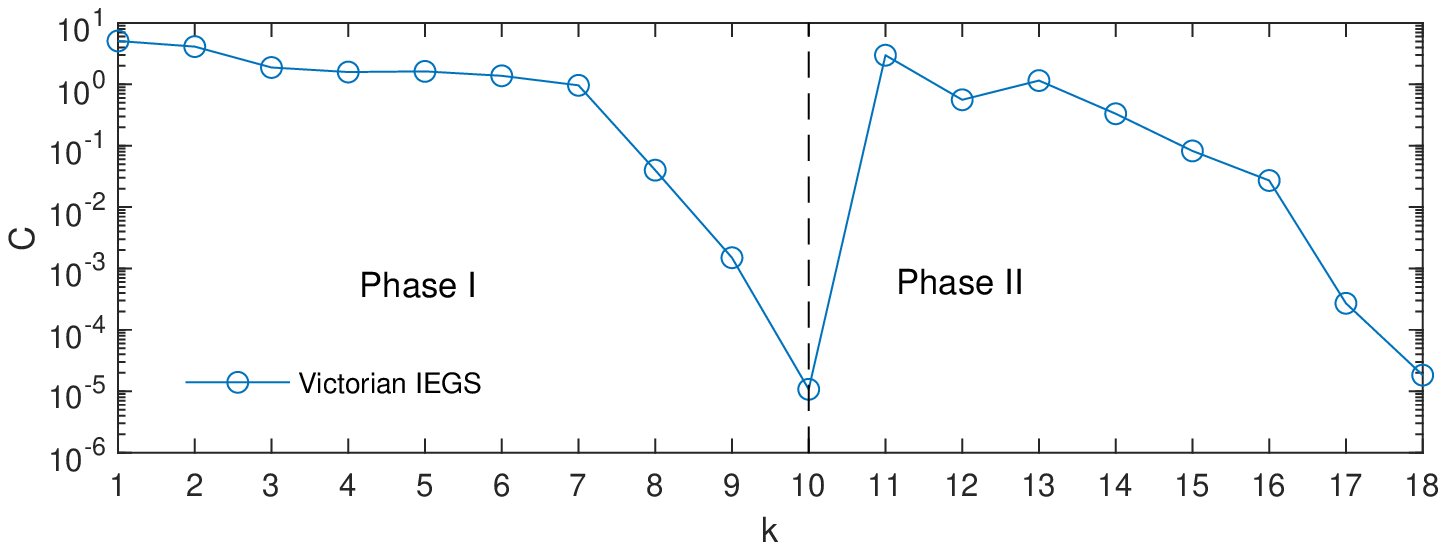}}
		\caption{Algorithm~\ref{Alg_SLPH}'s $C_{\rm max}^{k}$ vs $k$ for test case C3 under a warm start.}
		\label{fig_Max_viol_SLPH}
	\end{figure} 	
	Although it converges to slightly higher objective function values compared to Algorithm~\ref{Alg_SMILP}, the LP-based method in Algorithm~\ref{Alg_SLPH} is substantially faster. In particular, the speed-up is more than 8$\times$ on test case C3 under the same warm start strategy. And once again, the warm start strategy is generally noticeably faster than the cold start strategy, which underscores the value of using the solution of Problem~\ref{OEGF_Poly} as a starting point. Finally, it is worth noting that both Algorithms~\ref{Alg_SMILP} and \ref{Alg_SLPH} are exact only with respect to Problem~\ref{OEGF}, which incorporates a quasi-dynamic gas flow model, and not with respect to the original partial differential equations (PDEs) describing the full dynamic flow of gas.

\subsection{Linepack backtesting}\label{Sec_Linepack}
	
	To assess the validity of the solutions of Test cases C1, C2, and C3 from a physical gas network standpoint, they are backtested against actual 24-hour total linepack profiles for the state of Victoria, Australia. The 24-hour total linepack profiles of both Algorithm~\ref{Alg_SLPH} and the actual one from (AEMO) \cite{AEMO} are shown in Fig.~\ref{fig_Linepack} for the high demand day (09/08/2019).
	\begin{figure}[t]
		\centering{
			\psfrag{t}{\footnotesize $t$ ($\SI{}{\hour}$) \normalsize}
			\psfrag{Linepack}{\footnotesize Linepack ($\SI{}{\tera\joule}$) \normalsize}
			\psfrag{Actual}{\footnotesize Actual (AEMO) \normalsize}
			\psfrag{Algorithm}{\footnotesize Algorithm~\ref{Alg_SLPH} \normalsize}
			\includegraphics[width=\columnwidth] {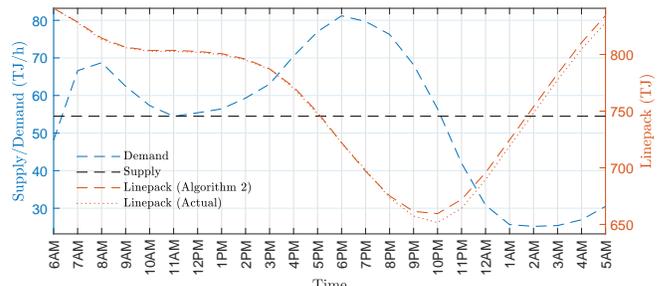}}
		\caption{Linepack backtesting for Test case C3.}
		\label{fig_Linepack}
	\end{figure}
	The maximum and mean errors between the two are $1.21\%$ and $0.38\%$, respectively, which could be attributed to the fact that the actual linepack on that high-demand day was not a result of optimized operation. The linepack backtesting not only serves as a testament to the validity of the solution of Algorithm~\ref{Alg_SLPH}, but also to the gas network model of the Victorian DTS which was developed from scratch as part of this work.
	
\section{Conclusion}\label{Sec_Conclusion}

	This paper introduced two novel SLP-based algorithms for efficiently solving the MOEGF problem. The first is an iterative MILP-based algorithm and the second is an iterative LP-based algorithm with an elaborate procedure for ensuring an integral solution. Numerical evaluation demonstrates that both algorithms can solve the MOEGF problem to high-quality \emph{feasible} solutions in computation times at least \emph{two orders of magnitude faster} than both a state-of-the-art NLBB MINLP solver and an MICP relaxation, on the \emph{real} electricity and gas transmission networks of the state of Victoria with actual linepack and demand profiles. Moreover, both approaches are warm-started from the solution of a novel polyhedral relaxation of the problem for a noticeable improvement in computation time compared to a cold start. Finally, while not claiming it is superior to existing MINLP solvers, the proposed iterative LP-based method represents a fast alternative for the specific type of problem addressed here, especially in settings where only LP solvers can be used (as is the case with existing dispatch engines of ISO like AEMO). To do so, the iterative LP-based method exploits the structure and engineering properties of the MOEGF problem and tailors a novel SLP approach, compounded by a fast heuristic that ensures an integral solution.

\section*{Acknowledgment}

	This work is supported by Future Fuels CRC as part of the R.P1.1-02: ``Regional Case Studies on Multi-Energy System Integration'' project. The cash and in-kind support from the industry participants is gratefully acknowledged.

\bibliographystyle{IEEEtran}
{\footnotesize
	\bibliography{SLPforIEGS}}

	\begin{IEEEbiography}[{\includegraphics[width=1in,height=1.25in,clip,keepaspectratio]{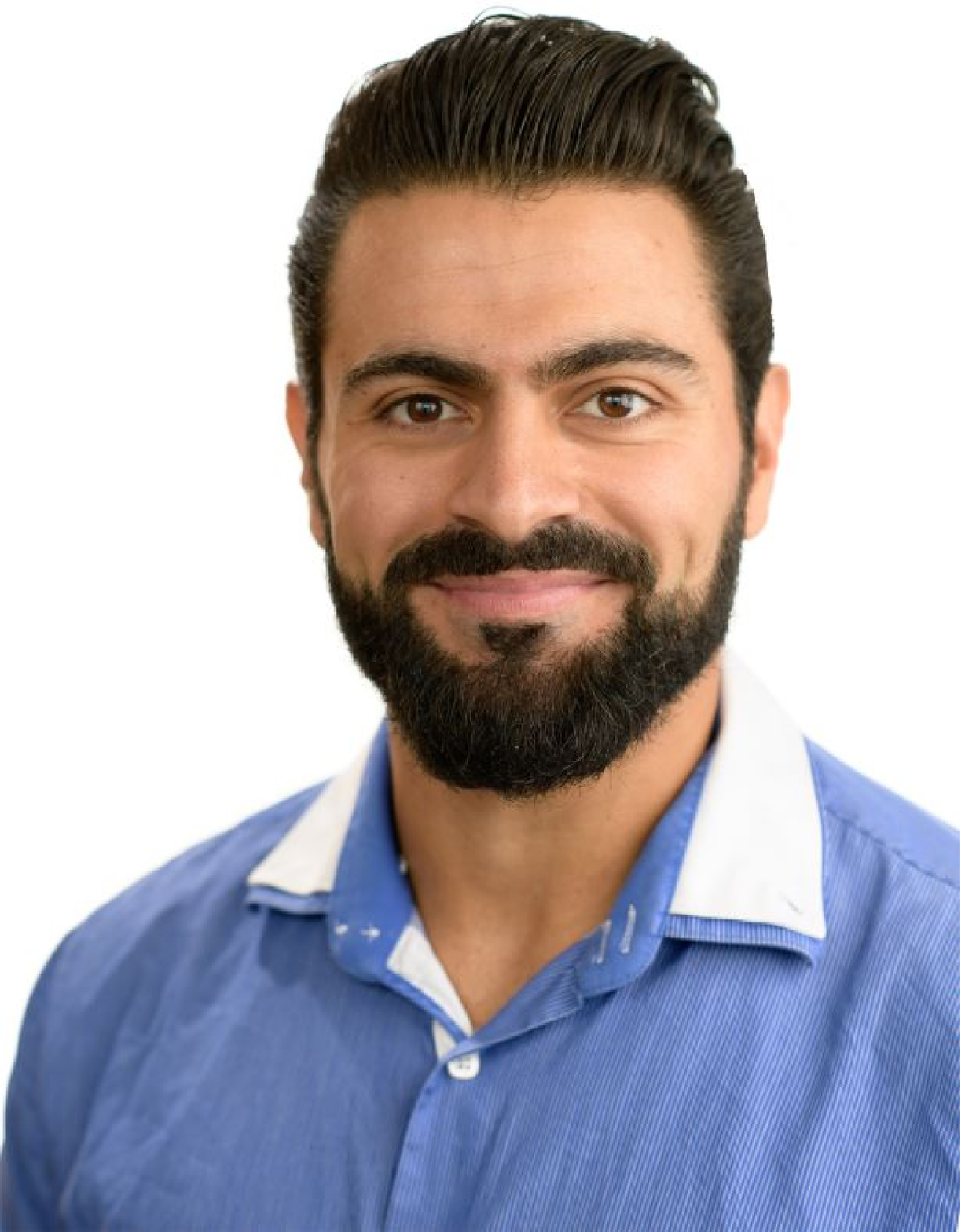}}]{Sleiman Mhanna}
		(S'13--M'16) received the B.Eng. degree (with high distinction) from the Notre Dame University, Lebanon, and the M.Eng. degree from the American University of Beirut, Lebanon, in 2010 and 2012, respectively, both in electrical engineering. He received the Ph.D. degree from the School of Electrical and Information Engineering, Centre for Future Energy Networks, University of Sydney, Australia, in 2016. He is currently a Research Fellow at the Department of Electrical and Electronic Engineering, The University of Melbourne, Australia. His research interests include computational methods for integrated multi-energy systems, decomposition methods, and demand response.
	\end{IEEEbiography}
	
		\begin{IEEEbiography}[{\includegraphics[width=1in,height=1.25in,clip,keepaspectratio]{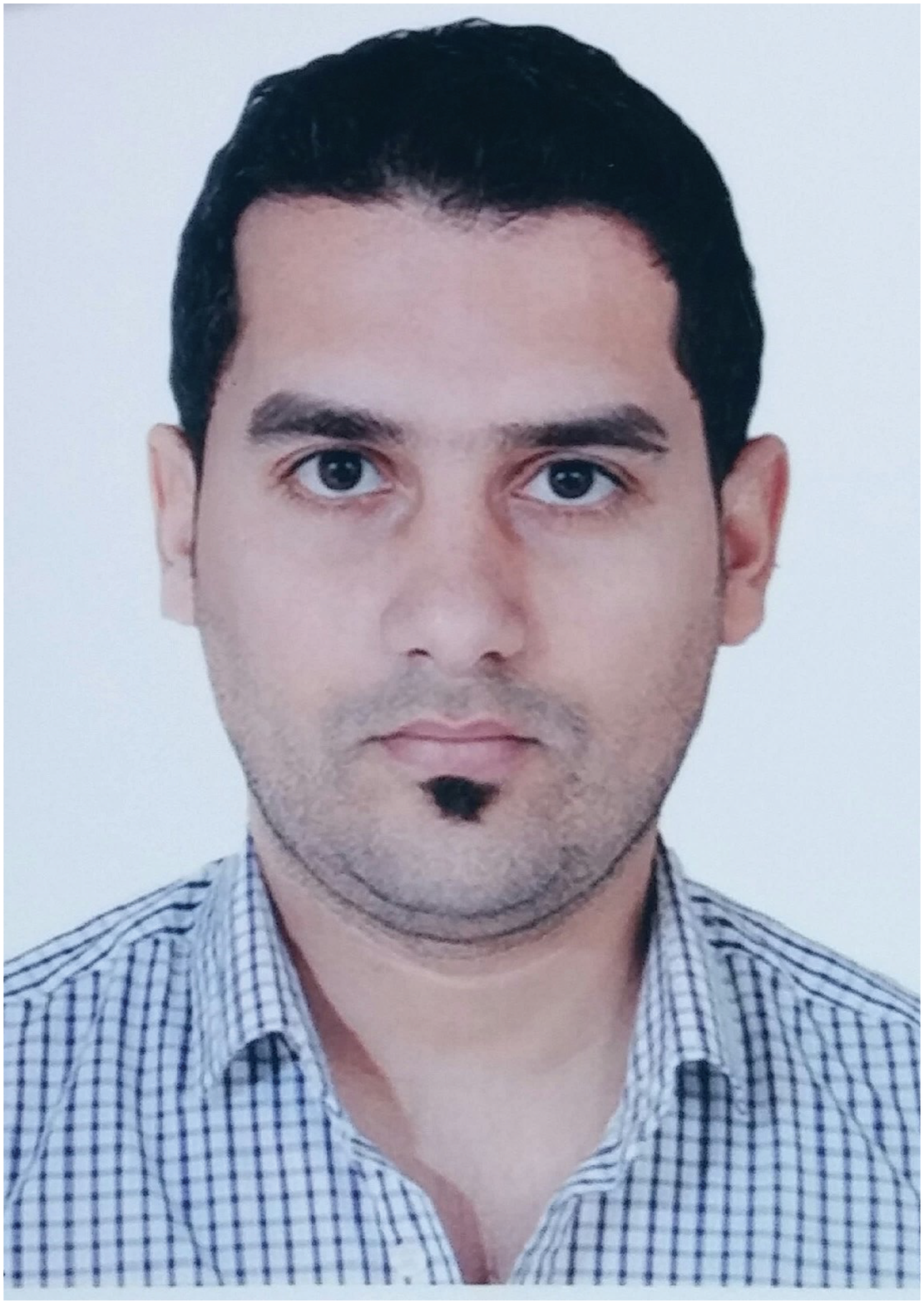}}]{Isam Saedi} (S'20) obtained the master’s degree in electrical power systems engineering from The University of Manchester, U.K., in 2015. He is currently in pursuit of his Ph.D. degree at the Department of Electrical and Electronic Engineering, The University of Melbourne, Australia. His research focuses on the assessment of integrated electricity-gas-hydrogen systems in the presence of different coupling technologies and scenarios for different sectors.
	\end{IEEEbiography}
	
	\begin{IEEEbiography}[{\includegraphics[width=1in,height=1.25in,clip,keepaspectratio]{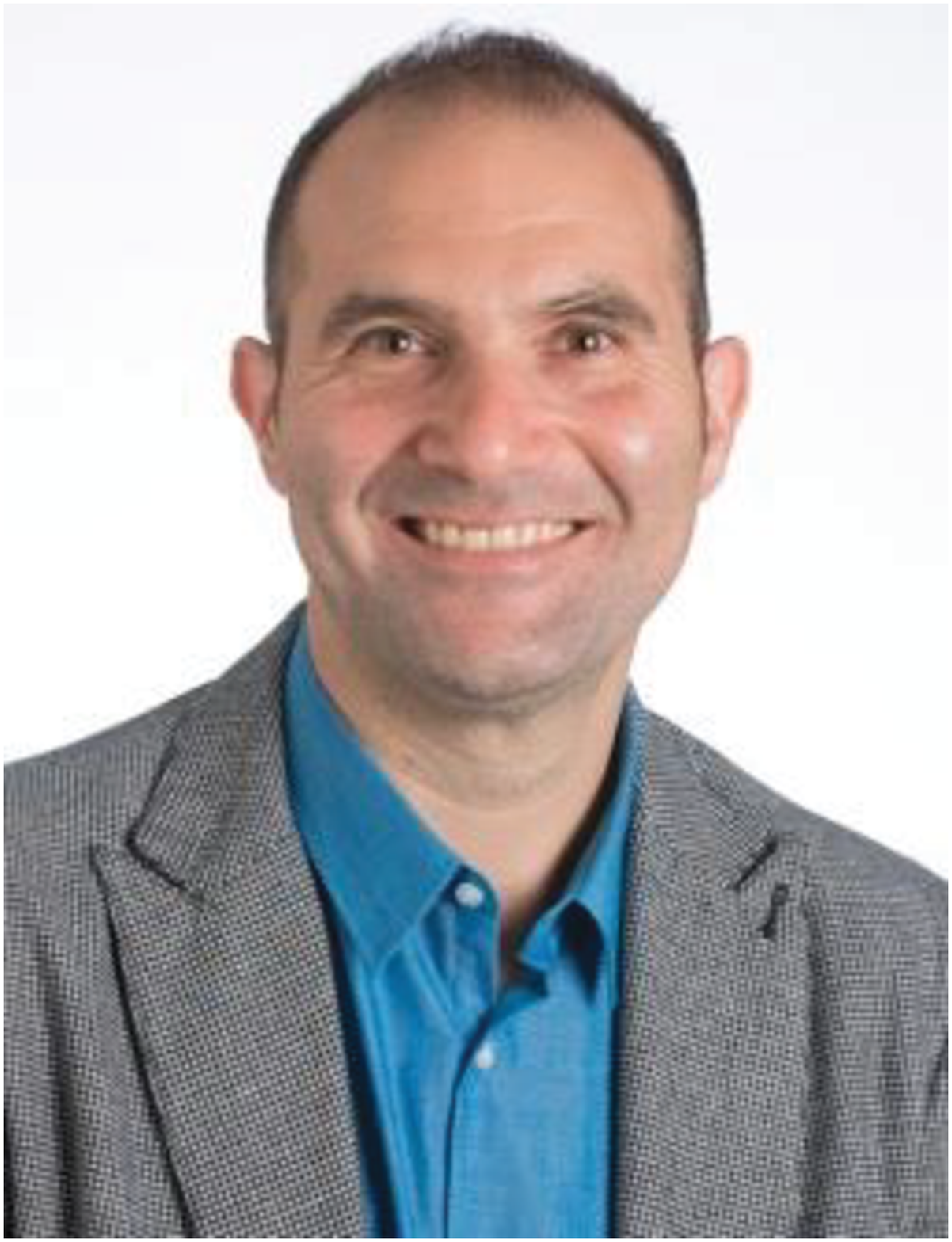}}]{Pierluigi Mancarella}
		(SM) is Chair Professor of Electrical Power Systems at The University of Melbourne, Melbourne, Australia, and Professor of Smart Energy Systems at The University of Manchester, Manchester, UK. His research interests include techno-economic modelling of integrated multi-energy systems; security, reliability and resilience of future networks; and energy infrastructure planning under uncertainty. Pierluigi is an Editor of the IEEE Transactions on Power Systems, IEEE Transactions on Smart Grid, and IEEE Systems Journal, and an IEEE Power and Energy Society Distinguished Lecturer.
	\end{IEEEbiography}

\end{document}